\documentclass[12pt,a4paper]{article}

\usepackage{amssymb}
\usepackage{amsmath}
\usepackage{amsthm}
\usepackage{graphicx}
\usepackage{xcolor}
\usepackage[mathscr]{euscript}
\usepackage[pdftex,hidelinks,bookmarks=false]{hyperref}
\usepackage{tikzsymbols}
\usepackage{tikz}
\usepackage{float}
\usepackage{hyperref}

\newtheorem{thm}{Theorem}
\newtheorem{lema}[thm]{Lemma}
\newtheorem{prop}[thm]{Proposition}

\newtheorem{cor}[thm]{Corollary}
\newtheorem{stmt}[thm]{Statement}
\newtheorem*{bandrepeat}{Lemma \ref{band}}

\theoremstyle{definition}
\newtheorem{defn}[thm]{Definition}
\newtheorem{question}[thm]{Question}

\begin{document}

\title{A Density Method for Bounding Heesch Numbers in $\mathbb{E}^3$}
\author{Aleksa D\v{z}uklevski \\
{\small\itshape Department of Applied Mathematics, Faculty of Mathematics and Physics,} \\
{\small\itshape Charles University, Czech Republic} \\[0.4em]
{\small\ttfamily aleksa@kam.mff.cuni.cz}}
\maketitle
\begin{abstract}
    In this paper, we provide a novel approach to limiting Heesch numbers of a specific family of tiles. In particular, we show how a certain family of three--dimensional prisms has a finite Heesch number. In so doing, we initiate a structural study of ratio $q$ tilings by two prototiles, in which, loosely speaking, one tile appears $q$ times more often than the other. The method works as follows: we take planar intersections with coronas around the prism in two perpendicular directions and compare the frequencies of the resulting prototiles in those patches. If every tiling by these prototiles has limiting ratio at least $q >1$, then sufficiently large coronas of the prism would require incompatible tile counts in the two directions. We also provide some computational results and exhibit $6$ infinite families of notched prisms that have finite Heesch numbers greater or equal to $2$. Finally, we  provide some negative computational results about the non--existence of certain prisms with given Heesch numbers.
\end{abstract}



% \subsection{To Do}
%     Could I give a numeric value on the upper bound?

%     dodaj u notation ili preliminaries da je min area a max A, da je diameter D. dodaj da je t(patch) broj plocica u patchu

%     dodaj i keywords i klasifikaciju

%     yappuj i kako je dosta plocica u dve korone i poredjenje sa ravni i sesticom

%     make sure that the || roudn the union some lines before isn't that long

%     make statements and lemmas that are reeated have the same number

%     nadji novi font za graph balls iz statement 3 mislim

%     dodaj komentar za dobijanje ovih protoskupova koji su ratio q. i spomeni da zapravo ne znamo primere yikes. mogu se dobiti dibalansiranima ali za njih vec imamo dokaz pa nije toliko zanimljivo. medjutim za generalizaciju iz teoreme 7 nista nismo znali dosad.

\section{Introduction}

Heesch's tiling problem, while one of the most prominent open problems in the mathematical theory of tilings, concerns itself with tiles that \textit{do not} tile the plane. In those cases, we are interested in measuring how close to a tiling a given tile can get, and as one measure of this notion introduce the Heesch number of a tile. Loosely speaking, it measures how many layers we can create around a tile. Crucial, therefore, is the concept of a \textit{corona} or layers around the tile. Formally, we define it as follows:

\begin{defn}
    We say that a tile $T$ can be \emph{surrounded $n$ times} if there exist finite collections of figures $\mathcal C_1,\mathcal C_2,\dots,\mathcal C_n$ such that:
\begin{itemize}
\item for each $i$, $1\leqslant i\leqslant n$, each figure from $\mathcal C_i$ is congruent to $T$;
\item every two different figures from $\{T\}\cup\bigcup_{i=1}^n\mathcal C_i$ have disjoint interiors;
\item for each $i$, $1\leqslant i\leqslant n$, each figure from $\mathcal C_i$ has a common boundary point with some figure from $\mathcal C_{i-1}$ (where, by convention, we let $\mathcal C_0=\{T\}$);
\item for each $i$, $1\leqslant i\leqslant n$, $\bigcup\left(\bigcup_{j=0}^i{\mathcal C_j}\right)$ contains $\bigcup\left(\bigcup_{j=0}^{i-1}{\mathcal C_j}\right)$ completely in its interior.
\end{itemize}

The collection $\mathcal C_i$ is called the \emph{$i^{\text{th}}$ corona}.

\end{defn}

An intuitive way of thinking about the $i^{th}$ corona is to think of it as the $i^{th}$ layer one can create around the tile. Since we are trying to make as big a patch of the tiling as possible, we want to locally mimic the restrictions that are required in a tiling (that is, that the interiors of distinct tiles are disjoined and that tiles exactly cover the whole plane). We also want to keep the coronas minimal, in the sense that we are not using any tiles that are not required in their formation. These ideas are mirrored in the previous definition; see Figure \ref{hes6} for an example. Observe also that the previous definition easily generalizes to sets of tiles, usually called \textit{protosets}.

It is worth noting that coronas are sometimes required in the literature to also be topological disks. This turns out to produce pathological examples of protosets with finite Heesch numbers that tile the plane \cite{basicdzuksliv2023}, and recently the same behavior has been reported by Ba\v si\' c, Kaplan and Mann \cite{tilinglist} for single tiles. We thus consider the stronger notion of a corona, in which we allow holes that could potentially be filled in, or at least surrounded at a later stage. All the proofs in the paper will work with this stronger notion of a corona, and will therefore be valid also for the weaker one.

We move on to define the Heesch number of a tile.

\begin{defn}
    The Heesch number of a tile $T$ is the biggest number of coronas that one can create around $T.$ If no such number exist, we say that the Heesch number of the tile is infinite.
\end{defn}

It is important to note that a Heesch number of a tile is infinite if and only if it tiles the plane. The first direction of this implication relies on the Extension theorem of Gr\"{u}nbaum and Shephard and is far from obvious, albeit intuitively easy to believe.

Heesch's tiling problem concerns itself with the following question, posed by Heesch in 1968 \cite{heesch1968}:

\begin{question}[Heesch's tiling problem]
    Is there an upper bound to all the finite values that can appear as Heesch numbers of tiles?
\end{question}

The record holders for the planar case, all obtaining Heesch number 6, are a family of tiles first discovered by Dave Smith \cite{smith2019}, as was later established by Ba\v si\' c and Bakker \cite{basic6, tilinglist}, together with an additional tile found in a search backed by AI; see Figure \ref{hes6}. In the same search, Bakker announced that another tile found by Smith has Heesch number 6.

\begin{figure}[ht]
  \center
   \includegraphics [width = 0.9\linewidth]{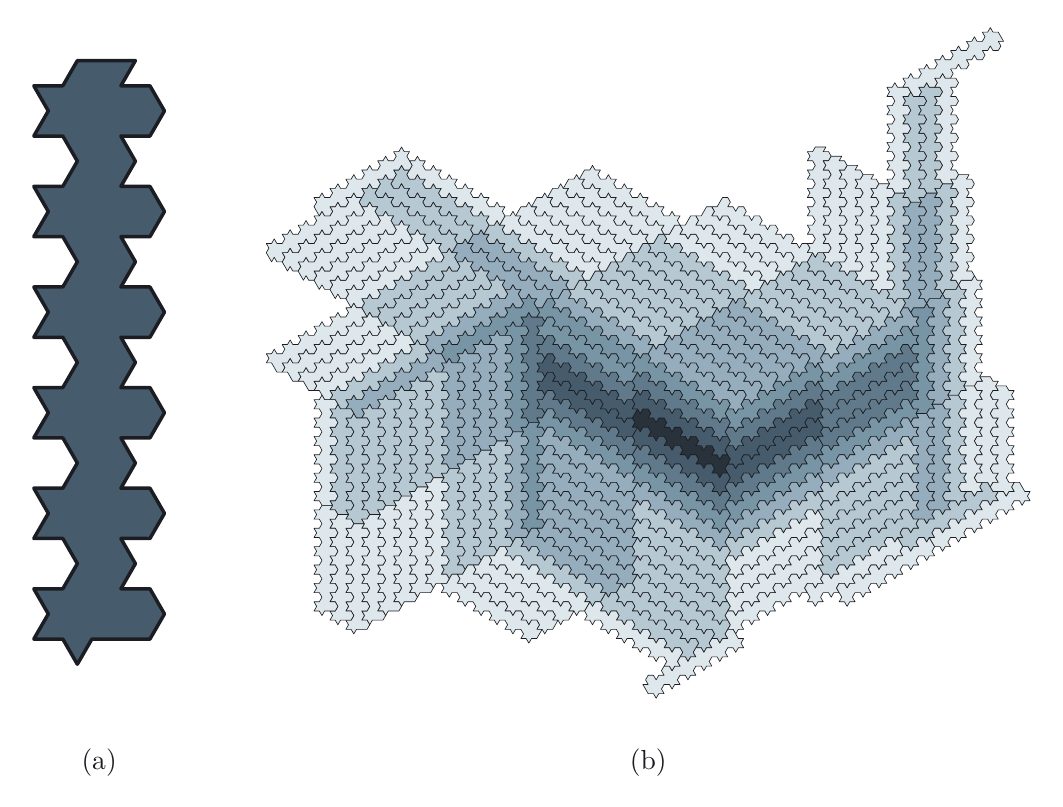} 
   \caption{A tile with Heesch number $6$ (a), together with its surrounding (b). Courtesy of Anton Bakker.}
   \label{hes6}
\end{figure}

It is often the case in tilings, and in other areas of mathematics and human endeavors in science generally, that relaxing one of the criteria of the problem leads to an easier version of that problem, which can sometimes yield to existing techniques or approaches, or new ones that would not be sufficiently strong for the original problem. One such weakening of the original Heesch's problem is to consider higher dimensional spaces and in particular the lowest example of $d=3.$ Additional dimensions supply more space for manoeuvring, as is seen, for instance, from Socolar's construction of tiles with arbitrarily large isohedral numbers, or by the constructions of sigma morphic and polymorphic tiles of Ba\v si\' c, D\v zuklevski and Slivkov\' a.

That said, not much work was done on Heesch's problem in $\mathbb{E}^3$ prior to this paper. Mann noted that a notched rhombic dodecahedron could produce tiles with high Heesch numbers \cite{mann2004}, and Bakker in his search produced examples with Heesch number 4. A notched cube with Heesch number 2 was also reported by Ba\v si\' c and Slivkov\' a \cite{basicsliv2022}, and in the same paper, they showed that taking notched cubes in arbitrarily large dimensions results in arbitrarily large Heesch numbers. This behavior is somewhat expected, but unfortunately doesn't solve the problem in any fixed dimension.

In this paper, we present a novel approach to the problem in $\mathbb{E}^3$, introducing a new method for bounding the Heesch number of a family of tiles. In doing so, we initiate the study of protosets that tile the plane in an unbalanced way -- meaning, intuitively, that one of the tiles is at least $q$ times more present than the other. We then discuss how this method could be used to potentially yield a disconnected tile in $\mathbb{E}^2$ with an arbitrarily large Heesch number and give some computational results.

During the preparation of this article, the construction of three-dimensional tiles with arbitrarily large Heesch numbers was announced by Ba\v si\' c \cite{tilinglist}. Despite that, we still find the results of this paper valuable as they present a novel approach to Heesch's problem that may yield future applications.

\section{Overview of the main ideas from the literature}\label{overview}

In this section we give a brief overview of all the main ideas that have arisen in the literature of limiting the Heesch number of a protoset (which often consists of just a single prototile).

Fontaine \cite{fontaine1991} was the first to discover a tile with a finite Heesch number bigger than 1. Her construction was somewhat ad hock but the main idea is to force a spiral arrangement of tiles that would lead to an intersection in the third corona. This approach doesn't seem to generalize, but is worth noting.

The first general method of finding a family of tiles that have finite Heesch numbers was introduced by Robert Amman \cite{amman1991} and can be showcased by Figure \ref{ammanbre}. Observe that, and this is the main point, this hexagon has an uneven number of protrusions (called \textit{bumps}) and intrusions (called \textit{nicks}). Observe further that if two edges are touching, then they are either both flat or one has a bump and the other a nick -- in both cases having the same number of bumps and nicks on that edge. Thus, the excess nicks that accumulate during the formation of the coronas have to be situated along the boundaries. However, since the base shape of this tile is a regular hexagon, the coronas will follow the underlying hexagonal grid; c.f. Figure \ref{ammanbre}. The main insight is that the boundary of this configuration is linear in the number of coronas, whereas the number of tiles inside those configurations is quadratic. Since each tile contributes one excess nick, this will result in a quadratic number of excess nicks, and only linear space to accommodate them, which is going to break at the first moment the quadratic function overcomes the linear one. Doing the exact calculations gives that the Heesch number of this tile is at most 3, with a matching lower bound evident from the Figure \ref{ammanbre}.\footnote{Technically, since we allow for holes in coronas, the Heesch number of the Amman hexagon should be $4$. The author, however, finds this extremely ugly so we postpone this discussion for some other time. Holes in the $4^{th}$ corona are very small and could never be accommodated by the whole tiles, and perhaps rendering such coronas as invalid would allow for the best definition.}

\begin{figure}[ht]
  \center
   \includegraphics [width = 0.5\linewidth]{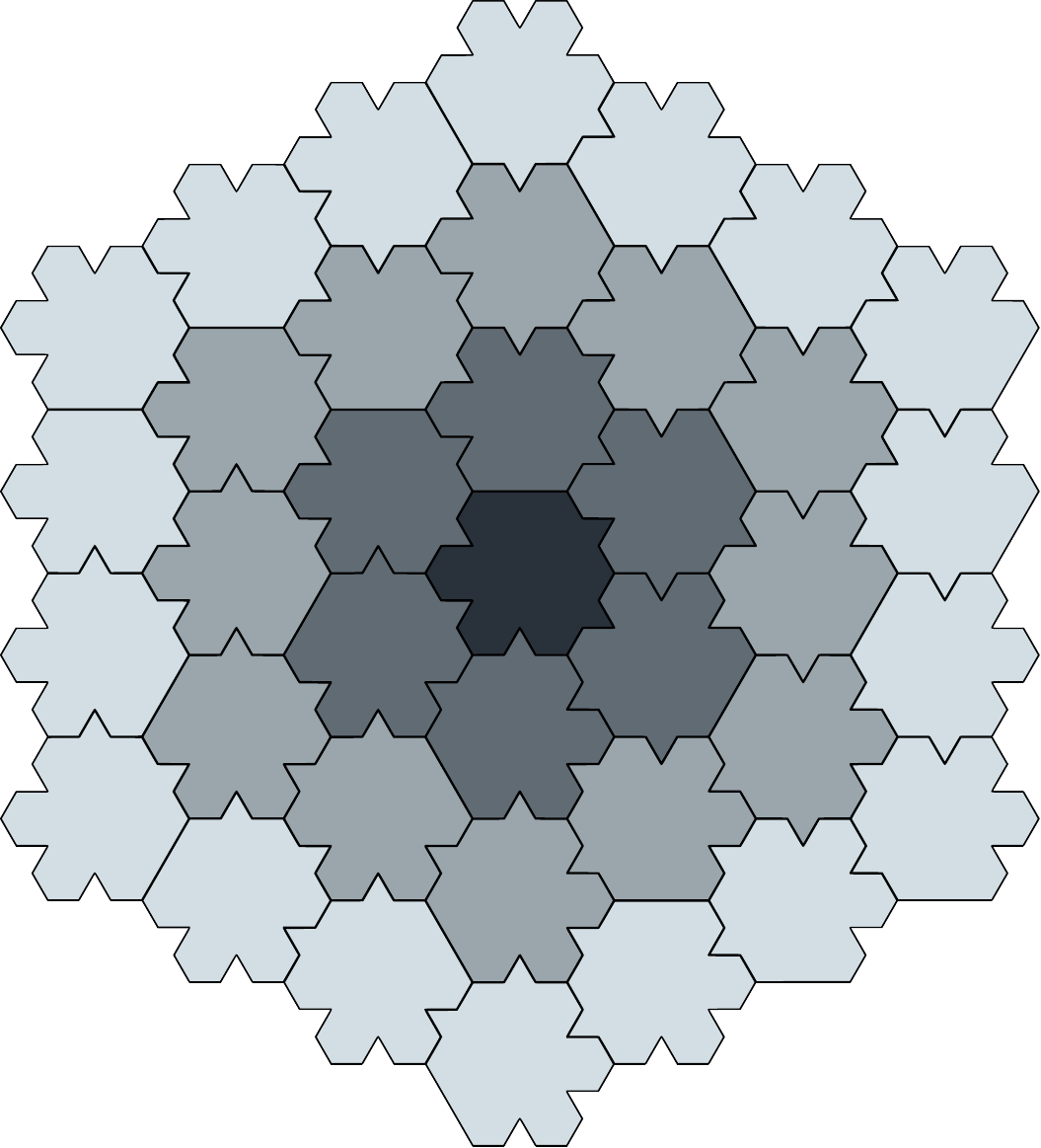} 
   \caption{Amman's hexagon surrounded three times.}
   \label{ammanbre}
\end{figure}

This very elegant method was later expanded upon by Casey Mann, who considered bigger polyhexes (specifically, stick tiles that consist of many hexagons that are glued together in a straight line) with notched edges. This resulted in his famous discovery of a tile with Heesch number 5 \cite{mannthesis,mann2004}.

Ba\v si\' c came up with a very nice observation \cite{basic6} that a tile living inside a vertex--transitive tiling by two regular polygons -- the Archimedean tiling of type $3.6.3.6$ (meaning that it has a triangle, hexagon, triangle and a hexagon again, meeting at each vertex in this cyclic order) -- and such that it has a disproportionate number of tiles of one kind to those of the other, has to have a finite Heesch number. The reason for this is that the tiles in the coronas have to follow the grid made by this tiling, but on average the tiling has twice as many triangles as hexagons. However, if a sufficiently big patch were to be covered by the observed tile, each tile would bring its own imbalance to the count, resulting in a big accumulated imbalance of one tile type to the other. This patch also has twice as many triangles as hexagons, which gives a contradiction, and we deduce that it cannot be covered.

% Kaplan reports same findings but without bigger Heesch numbers.

Ba\v si\' c, D\v zuklevski and Slivkov\'a \cite{basicdzuksliv2023} constructed a disconnected tile that showcases another idea: in it, two parts are at a distance of $2n+\frac{1}{2}$ and each corona brings two edges of those parts closer by $2$ units (one unit from each side). This, while allowing for the formation of $2n$ coronas, at the same time prohibits the formation of $2n+1$ coronas, due to non--integer distance between the starting pieces. This idea was later applied in the same paper to the sets of tiles.

A variation of this idea, yet somehow slightly of a different feel, is shown in Figure \ref{slicnen}. The goal here was to force a sort of circular structure that would end up overlapping with itself (perhaps somewhat reminiscent of Fontain's proof but not inspired by it). This structure can be realized by a protoset of two tiles, but doing so with a single tile is considerably harder. In Figure \ref{slicnen} we show a tile that can be surrounded $n$ times with copies of itself that are similar to it rather than congruent, and cannot be surrounded $n+1$ times. The apex angle, $\alpha$, made by extending the lateral edges of the underlying trapezoid, is chosen so that $\frac{2\pi}{2n+3}<\alpha <\frac{2\pi}{2n+1},$ and $\alpha \neq \frac{2\pi}{2n+2}.$

\begin{figure}[ht]
  \center
   \includegraphics [width = 8cm]{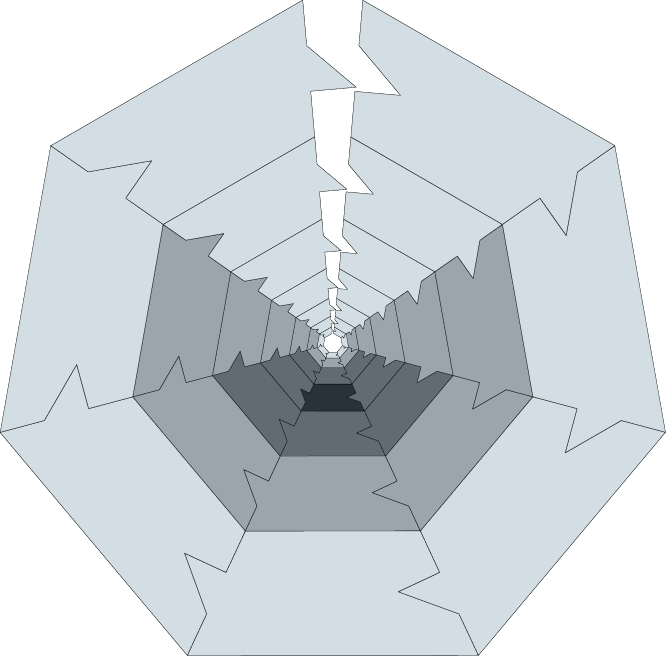} 
   \caption{A tile that can be surrounded $n$ times by tiles similar to it; illustration of the example $n=3.$}
   \label{slicnen}
\end{figure}

One idea of solving Heesch's problem in some higher, fixed dimension, would be to connect these two pieces using additional dimensions, and indeed this is what Ba\v si\' c announced \cite{tilinglist}. Not building upon such a set of two disjoint pieces, but from a different set of 4, with the same main idea.

In the hyperbolic space, Tarasov constructed tiles with arbitrarily large Heesch numbers \cite{tarasov2010}, and the idea was to force a chain of tiles that would all require the next tile to have a specific ordering of bumps and nicks along its boundary, which ultimately cannot be extended once we reach the $(n+1)^{st}$ corona.

This is, to the author's knowledge, a complete list of ideas that have so far been used. In what follows we exhibit a novel approach of limiting the Heesch number of a family of tiles.

\section{Motivation}

We will briefly explain the motivation behind considering the notched prisms in the first place.

In what follows, by a notched prisms we will consider a prism that is obtained by stacking cubes with bumps and nicks on their lateral faces, see Figure \ref{fig:perpendicular-sections} for an example. Note that while in the plane (i.e. when dealing with notched squares) we are restricted to using just a single shape for surrounding our tile, in space we could use several distinct shapes in each layer (along the $z$-axis), which would come from different layers of the prism, potentially ultimately allowing for a bigger number of coronas. Of course, the main obstacle to this approach is that layers are influenced by each other, stemming from the particular stacking of notched prisms of height one that comprise the big prism we are surrounding, but we were hopeful that clever ordering of the height one prisms could be found which would result in tiles with high Heesch numbers.

To support the claim that using several distinct types of notched hexagons (or squares) gives higher Heesch numbers (in a single layer), we found that a hexagon with a nick on each of its sides can have (at least) 4 coronas around itself made by Amman's hexagon, while using only Amman's hexagon results in just three coronas, as discussed. As in the plane, notched hexagons behave better with respect to Heesch numbers (reaching up to a Heesch number of 3, whereas the biggest Heesch number of a notched square is just 1), it was expected that they would also behave this way in three dimensions. Due to it being computationally more challenging to search for such tiles, we resorted to considering square--based prisms with notches on sides.

One way of looking at a notched prism with high Heesch number is to slice it along one of the axes parallel to its sides. This results in a surrounding of one of the notched rectangles using both tiles that are congruent to it, and those obtained by intersecting the prism with the plane parallel to the other vertical side of the prism. Doing the same gives a surrounding of the other rectangle. 

It is tempting to believe that therefore the protoset whose tiles are obtained by slicing the prism with these two planes has a finite Heesch number (even for the stronger definition of the Heesch number of a protoset, but since this is not the main topic of this paper we postpone this discussion to a future paper). However, and this is the crucial insight on which the main approach of this paper rests, this need not be the case -- in fact, the observed protoset could even tile the plane! The crucial fact that allows this behavior is that the layers parallel to the $x$--axis and the ones parallel to the $y$--axis are intertwined by the sheer fact that tiles in them share the base prism whose intersections they represent. And this is where the main idea comes from, as it represents one way of utilizing this interconnectedness between orthogonal layers. In particular, if the observed protoset were such that in any tiling admitted by it, the ratio of tiles  $\mathcal{P}_1$ to $\mathcal{P}_2$ is at least $q$, for some $q>1,$ then it would have a finite Heesch number. Think of notched rectangles $\mathcal{P}_1$ as yellow and of $\mathcal{P}_2$ as red (see Figure \ref{fig:perpendicular-sections}). If this were the case, and if we have sufficiently many coronas around the starting prism, then in each layer parallel to the $x$--axis (call them $x$--layers), we would have to have at least $q$ times more yellow tiles than the red ones. Thus, in total, looking at all the $x$--layers together, we would have at least $q$ times more yellow tiles than the red ones. The same has to hold for the $y$--layers. However, each yellow tile from an $x$--layer corresponds to a red tile in a $y$--layer! This is true since the prism they belong to has a red tile in its other direction. Thus, all the yellow tiles that we see in $x$--layers would be transformed into red tiles in $y$--layers, and therefore the majority of the tiles in the $y$--layers would have to be red; see Figure \ref{glavna} for a schematic representation. But this is impossible, as we cannot guarantee to have arbitrarily large patches in which red tiles would be predominant (if yellow is the one in surplus in all the tilings). This last statement needs a rigorous proof (see Lemma \ref{lem}), and we give a formal exposition of the whole idea in the next section.

As will be mentioned at the end of the next section, this approach will yield finite Heesch numbers to a broader family of tiles that behaves in this manner (Theorem \ref{generalizacija}), not just notched prisms.

\section{The Density Method}

Given a notched prism $\mathcal{P}$ we associate with it a set of tiles $\{\mathcal{P}_1,\mathcal{P}_2 \}$ where both $\mathcal{P}_1$ and $\mathcal{P}_2$ are obtained by intersecting $\mathcal{P}$ with planes parallel to its lateral sides, and orthogonal to the base of the bumps and nicks, passing exactly through the middle of each one they intersect; see Figure \ref{fig:perpendicular-sections} for an illustration. In the following we will consider only those prisms $\mathcal{P}$ for which the associated set $\{ \mathcal{P}_1, \mathcal{P}_2\}$ admits only tilings in which one of the tiles is "$q$ times more present", a notion that we shall formally define in Definition \ref{ratio}. Observe that if tile $T$ is $q$ times more present than tile $S$ in a tiling $\mathcal{T}$ then $S$ is $\frac{1}{q}$ times more present than tile $T$. In particular, unless they have equal ratio, at least one of the tiles is $l$ times more present in a tiling than the other, for some $l>1.$ We will call that tile tile $\mathcal{P}_1$ and the other one $\mathcal{P}_2$.

\begin{figure}[htbp]
  \centering
  \includegraphics[width=\linewidth]{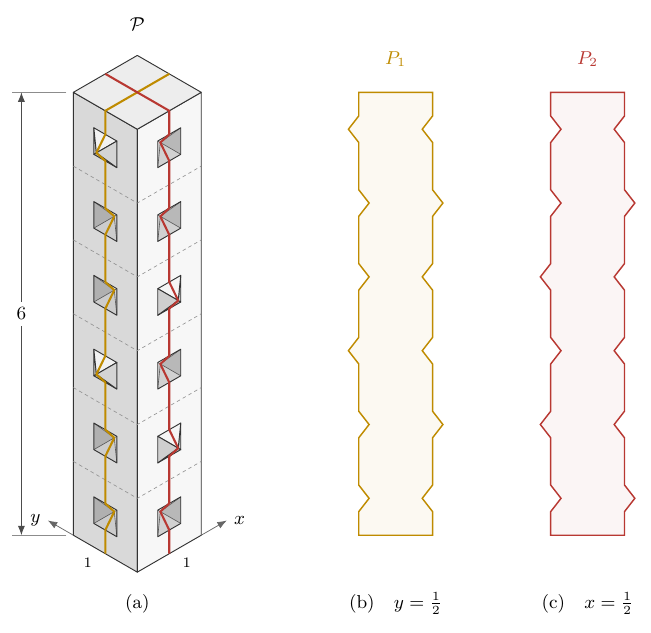}
  \caption{A notched prism and its two perpendicular central sections. The yellow and red intersections produce tiles $\mathcal{P}_1$ and $\mathcal{P}_2$.}
  \label{fig:perpendicular-sections}
\end{figure}

We begin with the definition of the ratio of tiles in a tiling.

\begin{defn}\label{ratio}
We say that a tiling $\mathcal{T}$ by two prototiles $\mathcal{P}_1$ and $\mathcal{P}_2$ is of ratio $q$ (of $\mathcal{P}_1$ w.r.t. $\mathcal{P}_2$), if $\lim_{r \rightarrow\infty} \frac{t_1(r,P)}{t_2(r,P)} = q$, for every point $P$. By $t_i(r,P)$ we denote the number of tiles congruent to $\mathcal{P}_i$ in the patch generated by the disk of radius $r$ centered at $P$. We further say that a protoset is of ratio $q$ if every tiling it admits has ratio $q$, or that it is of ratio at least $q$ if each tiling it admits has ratio at least $q.$
\end{defn}

We will show in Lemma \ref{welldef} that if the said limit exists for some point $P,$ then it exists for any other point $Q$, and is equal to it, but to prove that we will need Lemma \ref{lem}. Thus the definition could (and should be thought of as in this weaker version, for from it follows the stronger one) be restated with \textit{some} instead of \textit{every}. For now, however, we keep to Definition \ref{ratio}.

We state our main theorem:

\begin{thm}\label{thm1}
    Let $\mathcal{P}$ be a prism such that protoset $\{ \mathcal{P}_1, \mathcal{P}_2\}$ has ratio at least $q$, for $q > 1$. Then $\mathcal{P}$ has a finite Heesch number.
\end{thm}

We prove this theorem using the following Lemma, whose proof we postpone for later.

\begin{lema}\label{lem}
    For each $\delta > 0$ there exists an $n_{\delta}$ such that for all $n\geqslant n_{\delta}$ all the patches of $n$ coronas of the protoset $\{ \mathcal{P}_1, \mathcal{P}_2 \}$ have the ratio of $\mathcal{P}_1$ to $\mathcal{P}_2$ in $(q-\delta, +\infty).$
\end{lema}

Since the proof of this Lemma rather long and spans several Statements and Lemmas, we first show how it implies the statement of the Theorem.

\begin{proof}[Proof of Theorem \ref{thm1}.]

Suppose the contrary, that $\mathcal{P}$ can tessellate the space. That, in particular, means that we can create arbitrarily many coronas around $\mathcal{P}.$ We shall see how to choose $\delta$ later (depending only on $q$) but assume for now that we have chosen it to be sufficiently small. Lemma \ref{lem} gives us an integer $n_{\delta}$ and we take an arbitrary $n \geqslant n_{\delta}$. Since $\mathcal{P}$ tessellates the space, we can create $n$ coronas around it.

Taking the original prism, $\mathcal{P}$, to have its edges parallel to the $x$-- and $y$--axes, and its centroid (not taking into account the bumps and nicks) at $(0,0,0)$, we slice the structure of $n$ coronas by the planes that are orthogonal to and cut the $x$-- and $y$--axes at integer coordinates, thus obtaining $2(2n+1)$ planes with a nonempty intersection with $\mathcal{P}$. The intersections with the planes that cut the $y$--axis at integer points we will call the $x$--layers, and similarly for the $y$--layers. Observe that each layer in fact represents the surrounding of the prototiles $\mathcal{P}_1$ or $\mathcal{P}_2$ (in particular, $n$ coronas around them -- this is easily proven by induction; the important thing to note is that we have to force the tiles to all be vertically oriented. This is somewhat obvious, but we can make the height of each level of the prism, say, irrational, which would prohibit valid surroundings if one of the tiles was in horizontal orientation. Perhaps easier, the height of each layer could be made so small compared to the base of the prism that no tile could be oriented in a horizontal way). Consider, for instance, the $x$--layers. Due to the choice of $n$, each of the layers has the ratio of tiles $\mathcal{P}_1$ to $\mathcal{P}_2$ at least $q-\delta$. Looking at all the $x$--layers together, we obtain that in all of them jointly there are at least $q-\delta$ times more tiles $\mathcal{P}_1$ than tiles $\mathcal{P}_2.$ Notice that each prism in these $n$ coronas is cut into a $\mathcal{P}_1$ by a plane from one set (that is to say, in an $x$--layer or $y$--layer) and into $\mathcal{P}_2$ by a plane from the other set. What this in turn means is that in all the $y$--layers there will be at least $q-\delta$ times more tiles $\mathcal{P}_2$ than tiles $\mathcal{P}_1$, i.e. at most $\frac{1}{q-\delta}$ times more tiles $\mathcal{P}_1$ than tiles $\mathcal{P}_2$. Recall that we took the rectangle $\mathcal{P}_1$ to be one of the two possible, such that $q > 1.$ For $\delta < q-1$, we have that $q-\delta > 1$ and therefore $\frac{1}{q-\delta} < 1,$ and, in particular, $\frac{1}{q-\delta} < q-\delta$. Recall that jointly, in all the $y$--layers, there are at most $\frac{1}{q-\delta}$ times more tiles $\mathcal{P}_1$ than tiles $\mathcal{P}_2$. Therefore there exists a layer in which the ratio is smaller or equal than $\frac{1}{q-\delta}$ (for otherwise, in each layer we would have more than $\frac{1}{q-\delta}$ times many tiles $\mathcal{P}_1$ than tiles $\mathcal{P}_2$, which would, in turn, mean that in all the $y$--layers, looking at them jointly, there would be more than $\frac{1}{q-\delta}$ times more tiles $\mathcal{P}_1$ than tiles $\mathcal{P}_2$). Since $\frac{1}{q-\delta} < q-\delta$ by our choice of $\delta$, this is in contradiction with the choice of $n$ and this finishes the proof.

\end{proof}

\begin{figure}[!p]
  \vspace*{-1.5cm}
  \centering
  \makebox[\linewidth][c]{%
    \includegraphics[width=1.1\linewidth]{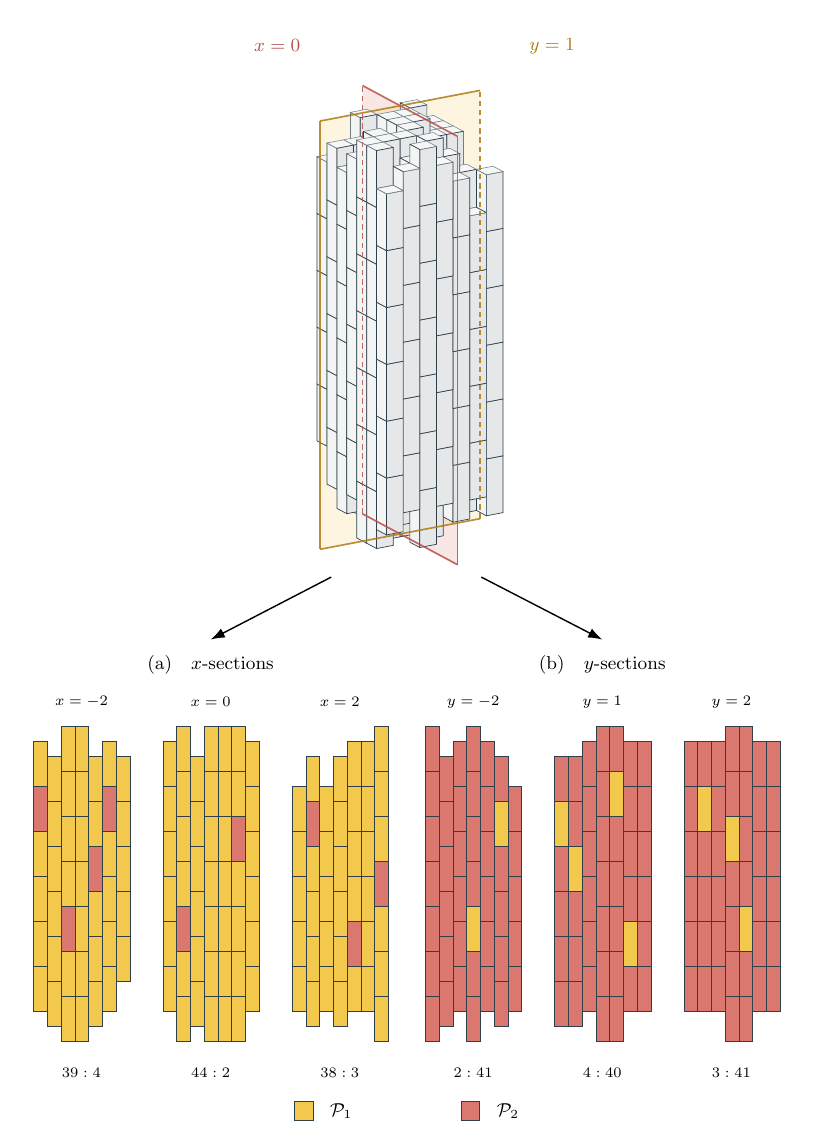}%
  }
  \caption{A schematic representation of the slicing argument.}
  \label{glavna}
\end{figure}

Note that Theorem \ref{thm1} readily generalizes. Indeed, for the same proof to work, we just need that the protoset obtained by intersecting the tile in both directions is disbalanced, and that the tiles in coronas are all vertical and follow a grid placement. Thus we get the following generalization:

\begin{thm}\label{generalizacija}
    Let $\mathcal{P}$ be a figure obtained by taking a square--based prism and gluing or subtracting some shapes along its four vertical sides. Let also the protoset $\left \{\mathcal{P}_1, \mathcal{P}_2 \right \}$ be obtained by intersecting $\mathcal{P}$ with planes parallel to the sides of the starting prism, and passing through the center of the original prism. Assume, furthermore, that the only way in which $\mathcal{P}$ can be surrounded is by copies of the form $g(\mathcal{P}) + (m\ell, n\ell, t),$ for $m,n\in \mathbb{Z}, t\in \mathbb{R},$ where $\ell$ is the length of the side of the base and $g$ preserves the vertical axis and acts on the horizontal plane as a symmetry of the square. Finally, assume that $\left \{ \mathcal{P}_1, \mathcal{P}_2 \right \}$ has ratio at least $q >1.$
    
    Then $\mathcal{P}$ has a finite Heesch number.
\end{thm}

We also note that an analogous theorem for the hexagon--based prisms holds.

Prior to proving Lemma \ref{lem}, we introduce some necessary notation.

For a fixed $q>1$ and $\delta >0$, choose $ q -\delta < q' < q$ which will present our lower bound on the ratio -- now every ratio $q$ tiling for some $r$ starts having ratio bigger than $q'.$ We say that a disk $\mathcal{B}$ is \textit{good} if $N_1 >  q'N_2$, and \textit{bad} otherwise, where $N_1$ is the number of tiles $\mathcal{P}_1$ that meet $\mathcal{B}$ and $N_2$ is the number of tiles $\mathcal{P}_2.$

Next, for a patch $\mathcal{A}_j$ we will denote by $\Omega_j$ the union of tiles in $\mathcal{A}_j$, which we will call its \textit{support} and by $|\Omega_j|$ their joint area.

By $u$ we will denote the \textit{inparameter} of $\left \{ \mathcal{P}_1, \mathcal{P}_2 \right \}$, meaning the biggest radius of a disk that can be inscribed in both $\mathcal{P}_1$ and $\mathcal{P}_2$, and by $U$ the \textit{circumparameter} -- the radius of a smallest disk in which both $\mathcal{P}_1$ and $\mathcal{P}_2$ can be inscribed. The smaller of the two areas of tiles in ${\mathcal{P}_1, \mathcal{P}_2}$, will be denoted by $a$, and the larger of the two areas by $A.$ Obviously, $a \geq \pi u^2$ and $A \leq \pi U^2.$ $D$ will stand for the bigger of the two diameters of tiles. The number of tiles in a patch $\mathcal{A}$ will be denoted by $t(\mathcal{A}).$

\begin{proof}[Proof of Lemma \ref{lem}]

We argue by contradiction. Suppose that we can find arbitrarily large bad patches $\mathcal{A}_j$ with $N_1 \leq bN_2$. The general proof strategy is the following: we will first show that any sufficiently big patch generated by a disk has to contain a good disk concentric to it. Then we show that for arbitrarily large patches, we can cover almost all of their area with good disks (meaning that we can find such disks inside of the patches). However, counting the tiles in those disks gives us that the original patch has good ratio, which is a contradiction. Along the way, we prove many technical lemmas and statements which have their own merit.

We first prove the following statement, which essentially tells us that we can always find a good small disk inside a big enough disk. The proof relies on the knowledge of the Extension theorem so we refer the reader to an exposition of it in \cite{gs1987,adams2022}

\begin{stmt}\label{stmt1}
    For each prescribed radius $L > 2U,$ there is a finite $T > L + U,$ s.t. every patch covering $\mathcal{B}_T(x)$ has a good disk centered at $x$ with radius between $L$ and $T-U.$
\end{stmt}

\begin{proof}[Proof of Statement \ref{stmt1}]

Assume the contrary, that no $T$ works. Then we can cover larger and larger disks about the origin but with no good concentric disks of the indicated size. Compactness argument from the Extension theorem of Gr\:{u}nbaum and Shepahard \cite{gs1987,adams2022} gives us a tiling from these patches. Since this is a tiling of the whole plane, it's ratio is at least $q$, so any sufficiently large disk will be good. Take one such disk centered at $0,$ $\mathcal{B}_r(0).$ Choose $r > L$ so that $\mathcal{B}_r(0)$ is good and no tile in the limit tiling is tangent to $\mathcal{B}_r(0)$. The reason for the latter requirement is that a tile just touching the boundary could result in it potentially not being counted in some of the patches from the converging subsequence. And indeed, the tiles in patches only converge to the ones in the tiling, and as such could move both away from and towards $\mathcal{B}_r(0),$ which could affect the count of the number of tiles in the patch. Take $R > r + D$. Now every tile touching $\mathcal{B}_r(0)$ is wholly contained in $\mathcal{B}_R(0).$ Thus by convergence on this larger disk $\mathcal{B}_R(0),$ we have that all the tiles touching $\mathcal{B}_r(0)$ in sufficiently late patches will have the same type and count as the tiles in the limit tiling.
However, for such patches we get a contradiction since a disk with the appropriate ratio inside them is good. And indeed, we can take $r > L$ in addition to the previous constraint, and choose $T$ such that the radius of the good disk, $r$,  is smaller than $T - U.$ This gives us the desired contradiction.
    
\end{proof}

Next we show that almost the whole patch can be covered by good disks.

To that end, we introduce $p$ as follows. For any disk size $L,$ we look at how much of the area of the whole patch $\mathcal{A}_j$, for $j$ sufficiently big, can be covered by good disks of radius at least $L.$ If for all sufficiently big $j$ we can cover at least the area of $a |\Omega_j|,$ then to $L$ we associate $a.$ Among all such values of $a,$ for all distinct $L,$ we take the supremum, and denote it by $p.$ Thus $p$ is the supremum of all $a\in [0,1]$ such that, for every $L>2U$, all sufficiently late patches $\mathcal A_j$ contain disjoint good disks of radii at least $L$ covering at least $a|\Omega_j|$.
% Described somewhat cumbersomely in plain language, $p$ is the greatest fraction of $|\Omega_j|$ that we can guarantee covering in every sufficiently late patch, using only disks whose radii exceed any prescribed lower bound.

Clearly, $0 \leq p \leq 1$, and we will show $p=1.$

Our ultimate goal is to count the tiles in these good disks to show that the whole patch is good. However, patches generated by disks can contain tiles not fully contained in their interior. Thus, a tile could be shared between several disks and counted with multiplicity in this way, which we have to take care of. One solution would be to not take those tiles into the final count for the starting patch, if we could show that the ratio of tiles in each shrunken disk remains relatively close to what it was before. This can indeed give us the adequate density of tiles in the final patch, but we find the following solution to be more elegant. What we will prove is that we can additionally require the disks to be at least $D$ apart (consider this to also be a part of the definition of $p$), which will make each tile belonging to one of the patches generated by one of them not belong to any other.

\begin{stmt}\label{stmt2}
    For any fixed $L > 2U,$ and $\eta > 0,$ every sufficiently late patch $\mathcal{A}_j$ has a finite collection of disks $\mathscr{B}_j$ such that:
    \begin{itemize}
        \item The disks lie in the patch and have pairwise distance greater than $D$;
        \item Every radius is at least $L;$
        \item Each disk is good;
        \item Together they occupy at least $(1-\eta)|\Omega_j|$ of the area of the patch.
    \end{itemize}
    In other words, $p=1.$
\end{stmt}

\begin{proof}[Proof of the Statement \ref{stmt2}]
We will prove that $p=1$ by improving any fraction below $1.$

To start, fix $L$ and take $T$ given by Statement \ref{stmt1}. By the definition of $p,$ sufficiently late patches have disjoint good disks of radius $M \gg T+D$ covering a fraction of $p-\varepsilon$ of their area. If $p=0$ we start with no disks.

Next we find new disks that will increase the total fraction of the good disks in the patch. As per Statement \ref{stmt1}, every uncovered point at distance at least $T$ from $\partial\Omega_j$ and at least $T+D$ from selected disks is the center of another good disk of radius between $L$ and $T-U.$ Note that this is not enough, as we have to prove that the fraction of the total area has increased, for which, for instance, taking a single additional good disk is not enough. Note also that $\mathcal{A}$ may contain holes, in which case $\partial \Omega$ contains both the outer boundary and the boundaries of those holes.

For the next step we need an obvious-looking geometric statement, whose proof we postpone for after this Lemma's. It essentially tells us that a boundary band of any fixed width occupies a vanishing fraction of the patch's area, for all patches of $n$ coronas. Formally,

\begin{lema}\label{band}
    For every fixed $w >0$ and every $\varepsilon > 0$, there is an $N$ such that, for every $n \geq N$ and every patch of $n$ coronas, $\mathcal{A},$ with support $\Omega,$ $$|\left \{ x: \text{dist}(x, \partial \Omega) < w \right \}| <\varepsilon \cdot |\Omega|. $$
 
\end{lema}

Call a point \textit{eligible} if it satisfies the above requirement of being sufficiently far away from both the boundary and existing disks. Then the set of all eligible points $E,$ is defined by $$E = \left \{ x \in \Omega_j: \text{dist}(x,\partial\Omega_j) > T, \, \,  \text{dist}(x, \text{old disks}) > T + D  \right \}.$$ As distance to a fixed set is a continuous function, for any point $x\in E$ we can find an open ball around it, meaning that $E$ is open. $E$ is thus measurable (we can represent it as a countable union of squares whose closures belong to $\text{int }E$ and since squares are measurable, and a countable union of measurable sets is still measurable, $E$ is as well).

Next we claim that all but at most $(O((T+D)/M) + o_j(1)) \cdot |\Omega_j|$ of the uncovered area belongs to $E.$  Indeed, the only parts of patches that are not eligible are the $T$-wide band around the boundary of the patch (which, according to Lemma \ref{band}, has a negligible area compared to the total, $o_j(\Omega_j)$), and the $T$-neighbourhoods outside of the disks. Each such annulus has an area at most $O((T+D)/M)$ times the area of the disk it is circumscribing, which gives us a total area of $O((T+D)/M)|\Omega_j|.$ Recall that we took $M \gg T+D$, so both types of non--eligible areas are indeed arbitrarily small. If we denote by $a_j$ the fraction of the patch $\mathcal{A}_j$ that our starting set of good disks, $\mathscr{B}_j$ has covered, we have $a_j \geq p-\varepsilon,$ so the lower bound on $|E|$ is given by 

\begin{equation}\label{eq1}
|E| \geq (1-a_j -O((T+D)/M) - o_j(1)) \cdot |\Omega_j|.
\end{equation} This bound will come in handy slightly later in the proof.

Now we choose disjoint disks from these candidates greedily. While we cannot guarantee the existence of a good disk with the largest possible radius (they could be approaching a supremum without ever reaching it), we can always choose a disk whose radius exceeds half of the supremum of the largest possible radii of the good circles that can still be added to our set $\mathscr{B}_j.$ We add that disk to our collection and discard all remaining candidate disks whose distance from it is at most $D$ -- we are keeping track of all the possible disks that can be added to our collection and just discard these from it (note, in particular, that we do not change $E$ in this process). We claim that the center of every discarded disk lies in the fourfold enlargement of the selected disk. Here is why. Assume that our greedy selection procedure selects a disk of radius $r$ and discards a candidate disk of radius $s$ because they intersect. Since $r > \frac{R}{2},$ and since the disk of radius $s$ was a candidate, it satisfied $s \leq R < 2r$. As the disks intersect, their centers are at a distance of at most $r + s + D\leq 3r +D < 4r$, since we took $r \geq L > 2U \geq D,$ which is what we wanted. Observe also that this greedy algorithm terminates in finite time since the patch is bounded and each disk has radius at least $L.$

Thus we have that $E \subseteq \bigcup_i 4D_i,$ where $D_i$ is the new collection of selected disks that we obtained from the algorithm. This gives: $$|E| \leq \sum_i |4 D_i| = 16 \sum_i |D_i| = 16\Bigl|\bigcup_i D_i\Bigr|,$$ where the last equality comes from disks being disjoint. Thus, using (\ref{eq1}), the new disks cover an area of at least $\frac{1}{16}|E| \geq \frac{1}{16}(1-a_j -O((T+D)/M) - o_j(1)) \cdot |\Omega_j|$, which will be almost $\frac{1}{16}(1-p).$ Together with the starting collection of disks, $\mathscr{B}_j,$ we would have a total percentage of almost $p + \frac{1}{16}(1-p)$ covered by good disks, which will be an increase, and a contradiction. However, there are a few error terms that we have to take care of.

Going back to the calculation, together with the starting good disks from $\mathscr{B}_j$, which cover an area of $a_j |\Omega_j|,$ and using that $a_j \geq p-\varepsilon,$ we have that at least the following percentage of $|\Omega_j|$ is covered by good disks: $a_j + \frac{1}{16} (1-a_j -O((T+D)/M) - o_j(1)) \geq p + \frac{1-p}{16} - \frac{15\varepsilon + O((T+D)/M) + o_j(1)}{16}.$ Term $\frac{1-p}{16}$ is a constant, and we can make the last term, $\frac{15\varepsilon + O((T+D)/M) + o_j(1)}{16}$ be smaller than it by choosing first a small enough epsilon, then taking $M \gg T+D$ and taking $j$ sufficiently large so that all the sufficiently large patches will have this term be smaller than $\frac{1-p}{16}.$ This ultimately gives us a contradiction since we have shown that we can cover a percentage larger than $p$ with good disks of any sufficiently large patch, which goes against the definition of $p$.

\end{proof}

Now we are ready to finish the proof. The main idea left is to take the area estimate of the good disks in the patch and transform it into tile count, which will ultimately give that the patch has a good ratio of tiles, a contradiction.

Choose the good disks from Statement \ref{stmt2}, $\mathcal{B}_{r_i}(x_i)$, with $r_i > L > D$, covering at least $(1 -\eta)|\Omega_j|$ area. Assign each tile that touches any of the disks to the disk it touches. Since we required in the statement of Statement \ref{stmt2} that disks be at a pairwise distance at least $D,$ this assignment is unique.
% Shrinking the radius of every such disk by $D,$ we obtain a collection of disks such that every tile touching any of them is fully contained in one of the original disks. Assign each such tile to the appropriate original disk it touches. By shrinking disks, we lose at most $\frac{2D}{L}$ of the total area: $\sum_i \pi (r_i^2 - (r_i^2 - 2r_iD + D^2)) = \sum_i \pi (2r_iD - D^2) \leq \frac{2D}{L}\sum_i \pi r_i^2 \leq \frac{2D}{L}|\Omega_j|.$

% Every unassigned tile lies outside all the shrunken disks. Therefore, the area occupied by unassigned tiles is at most $|\Omega_j| - \sum_i \pi (r_i - D)^2 = |\Omega_j| - \underbrace{\sum \pi r_i^2}_{\leq (1-\eta)|\Omega_j|} + \underbrace{\sum2\pi Dr_i}_{\leq \frac{2D}{L}|\Omega_j|} - \sum\pi D^2 \leq (\eta + \frac{2D}{L})\cdot |\Omega_j|.$

Next, let $N$ be the total number of tiles and $R$ the number of tiles left unassigned. Then the area estimate gives us: $R \cdot a \leq \eta|\Omega_j| \leq \eta\cdot N \cdot A$ and ultimately \begin{equation}\label{eq3}\frac{R}{N} \leq \frac{A}{a}\cdot \eta.\end{equation} A crucial observation is that we can make this fraction as small as needed by choosing some $L$ and a sufficiently small $\eta$. Then Statement \ref{stmt2} gives us $j$ large enough so that the area estimate we rely on here really holds. We thus make the fraction of unassigned tiles arbitrarily small.

Since each assigned collection is good, we have $N_{1,i} > q' N_{2,i},$ for a collection corresponding to $\mathcal{B}_{r_i}(x_i).$ Denote $m = N_{1,i} + N_{2,i}.$ Then $N_{2,i} < \frac{m}{q'+1}$, so we calculate $N_{1,i} - b N_{2,i} = m - (b+1)N_{2,i} > m(1 - \frac{b+1}{q'+1}) = \frac{q' - b}{q' + 1}\cdot m \equiv \gamma (N_{1,i} + N_{2,i}),$ for $\gamma = \frac{q'-b}{q'+1} > 0.$

Now we estimate $N_1 - b N_2.$ For assigned tiles, we have the previous bound, and each unassigned tile can contribute at worst $-b$ to $N_1 - bN_2$. Adding all the contributions, we get $$N_1 - bN_2 > \gamma (N-R) - bR = \gamma N - (\gamma + b)R.$$ We want the right-hand side to be positive to get a contradiction, so, equivalently, we want $\frac{R}{N} < \frac{\gamma}{\gamma + b}.$ Now choosing the fraction of the unassigned tiles to be below $\frac{\gamma}{\gamma + b},$ which we can do per (\ref{eq3}), we get a contradiction with $N_1 \leq bN_2.$ This finishes the proof.
\end{proof}

Finally, we prove Lemma \ref{band}:

\begin{bandrepeat}
    For every fixed $w >0$ and every $\varepsilon > 0$, there is an $N$ such that, for every $n \geq N$ and every patch of $n$ coronas, $\mathcal{A},$ with support $\Omega,$ $$|\left \{ x: \text{dist}(x, \partial \Omega) < w \right \}| <\varepsilon \cdot |\Omega|. $$ 
 
\end{bandrepeat}

\begin{proof}[Proof of Lemma \ref{band}]

We first prove the following Statement.
Prior to that, we introduce the \textit{contact graph} of the patch $\mathcal{A}.$ Contact graph $G$ of a patch $\mathcal{A}$ is obtained by taking a vertex for each tile from $\mathcal{A}$ and connecting two vertices if and only if the corresponding tiles touch, potentially only at one point.
\begin{stmt}\label{stmt3}
    Let $\mathcal{A}$ be a patch of $n$ coronas around tile $T_0.$ Denote by $\mathscr{B}_s(T)$ a ball in the contact graph: $$\mathscr{B}_s(T) = \left \{ T' \in \mathcal{A} : d_G(T, T') \leq s \right \}.$$ Then there exist constants $s_0 \in \mathbb{N}$ and $c,C >0$ which depend only on the protoset s.t. the following holds. For every $s_0 \leq s \leq n$ and every tile $T$ satisfying $d_G(T,T_0)\leq n-s$, i.e. for all the tiles at least $s$ layers from the boundary, $$cs^2 \leq |\mathscr{B}_s(T)| \leq C(s+1)^2.$$ In particular, the whole patch has $\Theta(n^2)$ many tiles.
 \end{stmt}

 \begin{proof}[Proof of Statement \ref{stmt3}]

 To see why the upper bound holds, observe that all the tiles corresponding to the vertices in a radius--$s$ ball in $G$ are contained in a disk of radius $(s+1)D$. Since each tile has area at least $a,$ we have at most $\frac{D^2\pi}{a}(s+1)^2$ many tiles. 

 To get the lower bound, observe first that any tile meeting a line segment of length $\ell$ lives in its $D$-neighbourhood. Total area of this neighbourhood is $2D\ell + D^2\pi$ and since tiles are disjoint, at most $\frac{2D\ell + D^2\pi}{a} = O(\ell+1)$ tiles can meet the segment.

 Tiles intersecting this segment form a path in $G.$ Therefore, a line of length $\ell$ meets at most $C_1\ell+ C_0$ tiles, for some constants $C_1, C_0.$ Taking $\alpha >0$ sufficiently small, so that $C_1 \alpha s + C_0 < s,$ we have that a line segment of length $\alpha s$ meets less than $s$ tiles. We next argue that there are no holes inside of the patch generated by $\mathcal{B}_{\alpha s}(x),$ for a point $x$ inside $T$. Let $y$ be any point at distance at most $\alpha s$ from $x.$ Suppose that the segment between $x$ and $y$ leaves the support $\Omega$ of $\mathcal{A},$ and stop at a first point $z$ on $\partial \Omega$. Since earlier coronas are completely surrounded, a tile $Q$ containing $z$ belongs to the last corona $\mathcal{C}_n,$ so $d_G(T_0,Q) = n.$ Then triangle inequality gives $d_G(T,Q) \geq n - d_G(T,T_0) \geq s.$ However, by the previous argument, the line segment from $x$ to $z$ meets fewer than $s$ tiles; these tiles make a path of length less than $s$ in $G$, i.e. $d_G(T,Q) < s,$ which is a contradiction.
 For every $y \in \mathcal{B}_{\alpha s} (x)$ line connecting $x$ and $y$ touches less than $s$ tiles, so the whole patch will be contained in $\mathscr{B}_s(T).$ Each tile has area at most $A$ so $t(\mathscr{B}_{s}(T)) \geq \frac{\pi \alpha^2}{A}s^2$.

 In particular, taking $s = n$ and $T = T_0$, we get that the total number of tiles is $\Theta(n^2).$
     
 \end{proof}

 Reflecting on the previous Statement, and in light of the pathological examples of protosets where holes can arise in valid coronas, which could later be extended, it is quite interesting that we can guarantee that a linear width around the starting tile will not contain any holes. To get this, take any $\alpha$ so that $C_1 \alpha n + C_0 < n$ has a solution for sufficiently large $n,$ i.e. such that $C_1 \alpha < 1,$ or $\alpha < \frac{a}{2D}.$ Then for $n > \frac{\pi D^2}{a -2D\alpha}$ the patch generated by $\mathcal{B}_{n\alpha}(x)$ will not contain any holes, for any point $x$ in the interior of $T_0$. What's more, we can easily turn this into an area estimate: denote by $r = \text{dist}(x, \partial \Omega).$ Then a disk centered at $x$ and of radius $r$ is hole free, and from the previous calculations, it meets at most $\frac{2D\ell +D^2 \pi}{a}$ tiles. Since it goes to the boundary, it meets $n+1$ tiles, so $n+1 \leq \frac{2D\ell +D^2 \pi}{a}$, which simplifies to $r \geq \frac{a}{2D}n - \frac{D^2 \pi - a}{2D} \equiv \frac{a}{2D}n - c',$ for a constant $c'$. In particular, the radius of this disk is linear in $n$ and so we have a disk whose area is a constant fraction of the whole (since the total area is on the order of $n^2$ -- which follows from the previous Statement) that is hole-free.  Since this bears its own merit, we state it as a corollary:

 \begin{cor}
     For any $0< \alpha < \frac{a}{2D}$ and any sufficiently large $n$, the disk $\mathcal{B}_{n\alpha}(x)$ is fully contained in any $n$-corona patch around $T_0,$ for any point $x$ in the interior of $T_0$. As a consequence, there is a disk which covers a fraction of the total area of the patch that is hole-free.
 \end{cor}

Next, we prove that the number of tiles in the last constant number of coronas is a vanishing portion of the whole.

\begin{stmt}\label{stmt4}
    For a fixed $h \geq 1,$ and any patch of $n$ coronas, $\mathcal{A},$ let $T_h$ represent the tiles in the last $h$ coronas (coronas $n-h+1, ..., n$). Then $$\lim_{n \rightarrow \infty} \frac{t(T_h)}{t(\mathcal{A)}}=0.$$
\end{stmt}

\begin{proof}[Proof of Statement \ref{stmt4}]

Choose an integer $s \geq \max\left \{ h, s_0\right \}$ s.t. $8s \leq n.$ Choose also a maximal family of tiles in $T_h$ whose pairwise graph distances in $G$ exceed $ 20s$. Due to their maximality, the radius--$20s$ balls cover $T_h.$ Statement \ref{stmt3} gives us that each such ball contains $O(s^2)$ many tiles. Denoting by $N$ the number of centers of distinct balls in $T_h,$ we get: $N \cdot c_1 \cdot s^2 \geq t(T_h)$, so \begin{equation}\label{eq2}
    N \geq \frac{t(T_h)}{c_1\cdot s^2} \equiv c_2 \frac{t(T_h)}{s^2},
\end{equation} 
for $c_2 = \frac{1}{c_1}$ constant.

Next, from each center we move $5s$ steps towards the root, along the shortest path. Take the radius--$s$ balls from these points, and denote them by $\mathscr{B}_s'.$ The new centers are still more than $10s$ apart, so the balls in $\mathscr{B}_s'$ are disjoint. Statement \ref{stmt3} gives that each such ball has at least $c\cdot s^2$ many tiles (and we can apply it since $s \geq s_0$), and all of them lie between depths $4s$ and $7s$ from the outermost corona (more precisely, between depths $4s$ and $6s + h-1$ since we start from the corona of depth at most $h-1$). Recalling (\ref{eq2}), this gives that they contain at least $$cs^2 \cdot c_2 \frac{t(T_h)}{s^2} \equiv \beta \cdot t(T_h),$$ for $\beta$ a constant, which is a positive fraction of the number of tiles in $T_h.$

Now we repeat this procedure with $t$ replacing $s$, for $t = 8s, 8^2s, 8^3s,\dots,$ as long as $8t \leq n.$ More precisely, we again choose a maximal set of centers in $T_h,$ pairwise separated by more than $20t,$ then move each center $5t$ steps inward and put a radius--$t$ graph ball around each new center. The resulting balls lie progressively deeper in the patch for bigger values of $t$, in bands of depth $[4t, 7t).$ These bands are pairwise disjoint since $4 \cdot 8^{i+1}s > 7 \cdot 8^is.$ Each again contains at least the same positive fraction of the tiles in $T_h$, $\beta \cdot T_h$. Since they all lie in the original patch, as $n$ increases, $t(T_h)$ must present a vanishing fraction of its total tile count.

More formally, let $m$ be the number of scales $s, 8s, 8^2s, ...$ for which the construction is repeated. At each step we obtain a band containing at least $\beta \cdot t(T_h)$ many tiles, for a constant $\beta$ independent of $t.$ Since these bands are pairwise disjoint and lie in $\mathcal{A},$ $t(\mathcal{A}) \geq m \cdot \beta \cdot t(T_h),$ and so $$\frac{t(T_h)}{t(\mathcal{A})} \leq \frac{1}{m \beta} \rightarrow 0$$ since $m \rightarrow \infty$ as $n\rightarrow \infty.$

\end{proof}

We are now ready to finish the proof of Lemma \ref{band}.

To that end, denote the band of width $w$ around and inside $\partial \Omega$ by $I_w$. Every point of $\partial \Omega$ belongs to a tile in the last corona $\mathcal{C}_n.$ Choose a point $z_T$ in each such tile $T.$ Since diameters of tiles are at most $D,$ $I_w \subseteq \bigcup _{T \in \mathcal{C}_n} \mathcal{B}_{w + D}(z_T).$
Consequently, $$|I_w| \leq \pi (w +D)^2 \cdot t(\mathcal{C}_n) = \pi (w +D)^2 \varepsilon_n \cdot t(\Omega) \leq \frac{ \pi (w +D)^2 \varepsilon_n}{a} \cdot |\Omega| = o(|\Omega|),$$ where we used Statement \ref{stmt4} with $h=1$ to write $t(\mathcal{C}_n) = \varepsilon_nt(\Omega),$ for $\varepsilon_n \rightarrow 0$ and the fact that tiles have area at least $a.$
    
\end{proof}

An easier attempt at a proof of Lemma \ref{lem} is to try to prove an analogue of the Extension theorem -- that if we have patches that are all bad then we will have a tiling that is bad as well. However, the following construction shows that the direct compactness argument from the Extension theorem doesn't work in this case. For each patch $\mathcal{A}_j$, denote by $R_j$ its radius, and fill the approximations of disks of radius $\frac{R_j}{10}$ with one type of squares -- say, with sides of size 2 -- and the rest of the disk with unit squares. Calculating the ratio gives that there are approximately $396$ times more unit squares than the size 2 squares in each patch. However, the tiling that is obtained by the compactness argument of the Extension theorem gives is the one in which the only tile that appears is square of size 2, which gives us the predominance of the other tile than in the patches (technically, we defined $\mathcal{P}_1$ to be the prevalent one, but when a protoset is fixed we define the ratio to be the limit of one of the two fixed ratios, which in this case gives us the exact opposite of what we wanted). However, this does not imply that the implication "arbitrarily large bad patches exist" $\Rightarrow$ "there exists a bad tiling" is false. One might be tempted to read Lemma \ref{lem} as saying the exact contrapositive of that -- that if a protoset has ratio at least $q$ then there cannot exist a sequence of arbitrarily large bad patches -- but this is not the case. What the Lemma states is that the ratios of all the patches cannot all be smaller than some fixed number smaller than $q,$ not that they cannot approach it from below. Thus, the status of the Extension theorem for ratio $q$ tilings is still unresolved.

However, requiring something slightly stronger gives somewhat of an analogue of the Extension theorem: if the only way a patch of radius $r$ can be formed is in such a way that it has ratio at least $q$ then the protoset will have ratio (if it exists) at least $q$. Indeed this is immediate since then any patch has ratio at least $q$ and we get that the limit is at least $q$. Note, however, that the limit need not exist in this case.

\subsection{Unique ratio}

In this section we prove that Definition \ref{ratio} can be weakened to say that the tiling is of ratio $q$ if for \textit{some} point $P$ $\lim_{r \rightarrow\infty} \frac{t_1(r,P)}{t_2(r,P)}=q $ as opposed to the previous restriction of it being true for any point $P.$ In other words, we show that two points in the same tiling have the same ratio if one exists. In order to prove this, we will need a modification of the Normality Lemma \cite[Theorem~3.2.2]{gs1987}. Before stating the theorem we recall that a tiling is said to be $\textit{normal}$ if: \begin{itemize}
    \item Each tile in it is a closed topological disk;
    \item The circumparameter and inparameter of the protoset of the tiling exist;
    \item The intersection of any two tiles in a tiling is a connected set.
\end{itemize}
It is easily seen that indeed any tiling by $\{\mathcal{P}_1, \mathcal{P}_2 \}$ satisfies these three criteria.

\begin{lema}[Modification of Normality Lemma]\label{lem:norm}
If $\mathcal{T}$ is a normal tiling by $\mathcal{P} = \left \{\mathcal{P}_1, \mathcal{P}_2 \right \}$, and if $\mathcal{T}$ has ratio $q$ at $P,$ then for arbitrary $x > 0:$ $$\lim_{r \rightarrow \infty} \frac{t_1(r+x,P)}{t_1(r,P)} = 1.$$ Equivalently, \begin{equation}\label{norm}
    \lim_{r\rightarrow \infty} \frac{t_1(r+x,P) - t_1(r,P)}{t_1(r,P)} = 0.
\end{equation}
    
\end{lema}

\begin{proof}
        We will show the equation (\ref{norm}). Denote by $u$ the inparameter of the tiling $\mathcal{T}$ and by $U$ the circumparameter. Then, since the patch of tiles generated by $\mathcal{B}_{r}(P)$ contains (by definition) $\mathcal{B}_{r}(P),$ and since the area of each tile is at most $U^2\pi$, we have that $t(r,P) \geqslant \frac{r^2\pi}{U^2\pi} = \left (\frac{r}{U}\right )^2,$ for arbitrary $r > 0.$ Now we make use of Lemma \ref{lem}. Tako some $\delta,$ and take $r_{\delta}$ to be the smallest radius such that each union of the first $n_{\delta}$ coronas around any of $\mathcal{P}_1$ or $\mathcal{P}_2$ is contained in $\mathcal{B}_{r_{\delta}}(P_0),$ where $P_0$ is the centroid of the zeroth corona (which is just a single tile $\mathcal{P}_1$ or $\mathcal{P}_2$), again not accounting in the bumps and nicks. Lemma \ref{lem} then tells us that $\frac{t_1(r,P)}{t_2(r,P)} \geqslant q-\delta,$ for any $r > r_{\delta}.$ Since $t_1(r,P) + t_2(r,P) = t(r,P),$ we have that $t_1(r,P) \geqslant \frac{q-\delta}{q-\delta+1}t(r,P) \geqslant \frac{q-\delta}{q-\delta+1}\left (\frac{r}{U}\right )^2.$ On the other hand, $t_1(r+x,P)-t_1(r,P) \leqslant \frac{(r+x+2U)^2\pi - r^2\pi}{u^2\pi} = \frac{\pi(x+2U)(2r + x + 2U)}{u^2\pi},$ since we are counting the number of tiles in the annulus between two discs. The last two inequalities together give $$\frac{t_1(r+x,P)-t_1(r,P)}{t_1(r,P)} \leqslant \frac{(x+2U)(2r + x + 2U)U^2}{\frac{q-\delta}{q-\delta+1}u^2r^2},$$ which tends to zero as $r \rightarrow \infty$ ($\frac{q-\delta}{q-\delta+1}$ is a constant).
\end{proof}

% Note that in order to prove the analogous statement for $t_2(r,P)$ we would need to have an upper bound on the ratio $q$ in each tiling. And indeed, assuming that we consider only protosets in which no tile tiles the plane (and it is not obvious that for our goal we shall not try those as well), the argument to show that boundedness could go as follows: take a very large $K,$ and consider some radius $r_K$ such that in a patch from tiling $\mathcal{T}_K$ we have the ratio of tiles in some $\delta$-neighbourhood of $K.$ Circumscribe each tile $\mathcal{P}_2$ by a circle of radius $U$. Next, we consider the densest packing of circles of radius $U$ into a large circle of radius $r$ (such that the percentage of their area to the area of the large circle is adequate), and consider a circle that does not intersect any of these smaller circles. A patch generated by it will necessarily contain only tiles congruent to $\mathcal{P}_1$, and taking ever larger ratio $K$, in this way we obtain larger and larger patches made only of tiles congruent to $\mathcal{P}_1,$ and as a consequence of the Extention theorem, we obtain that $\mathcal{P}_1$ tiles the plane, which is a contradiction. In the other case, indeed no such bound need exist, and the analogue of equation (\ref{norm}) need not hold.

We are finally ready to prove that the ratio is equal in any point in the tiling.

\begin{lema}\label{welldef}
    If in any tiling $\mathcal{T}$ by two prototiles $\mathcal{P}_1$ and $\mathcal{P}_2$ $\lim_{r \rightarrow\infty} \frac{t_1(r,P)}{t_2(r,P)} = q$, for some point $P$ in the plane, then for any other point $Q$ we have $\lim_{r \rightarrow\infty} \frac{t_1(r,Q)}{t_2(r,Q)} = q.$
\end{lema}

\begin{proof}
    We follow a similar line of reasoning as that in \cite[Theorem~3.3.2]{gs1987}. Let $P$ and $Q$ be arbitrary, and denote by $\frac{x}{2}$ the distance between them. Then obviously $\mathcal{B}_{r}(P) \subset \mathcal{B}_{r+\frac{x}{2}}(Q) \subset \mathcal{B}_{r+x}(P),$ and therefore the analogous relations hold for the patches generated by them. It then follows that $t_1(r,P) \leqslant t_1(r+\frac{x}{2},Q) \leqslant t_1(r+x,P)$ and $t_2(r,P) \leqslant t_2(r+\frac{x}{2},Q) \leqslant t_2(r+x,P)$. The following inequalities are then easily checked to be true:

    \begin{equation}
        \frac{t_1(r+x,P)}{t_2(r+x,P)}\cdot \frac{t_1(r,P)}{t_1(r+x,P)} \leqslant \frac{t_1(r+\frac{x}{2},Q)}{t_2(r+\frac{x}{2},Q)} \leqslant\frac{t_1(r+x,P)}{t_2(r,P)}\cdot \frac{t_1(r,P)}{t_1(r,P)}.
    \end{equation}

    Taking the limit of both sides as $r\rightarrow \infty$, and using Lemma \ref{lem:norm}, we get that both the leftmost and the rightmost side is equal to $q$, proving the claim. 
\end{proof}

\section{Application to $\mathbb{E}^2$}

In this section we briefly show how the approach from the previous section can be utilized in the two-dimensional case.

Applying the density method to the two-dimensional tile results in one-dimensional tiles (which are often considered as unions of intervals) that play the roles of $\mathcal{P}_1$ and $\mathcal{P}_2.$ Indeed, any intersection that we get has to lie on a line. One of the obstacles in this case is that the intersection of a line with a tile is often just a single interval, which tiles the line on its own. Making this intersection disjoint calls for parts of the tile which would serve as "connectors" between different parts from the intersection, which makes the task of surrounding the tile substantially more complicated. We therefore restrict our attention to disconnected tiles in $\mathbb{E}^2$. This variation of the problem has already been solved by Ba\v si\' c, D\v zuklevski and Slivkov\' a, \cite{basicdzuksliv2023}, but we discuss how the density method could be potentially applied to obtain disconnected tiles with arbitrarily large Heesch numbers. One could then hope to somehow easily connect these pieces along the third dimension.

The protoset described below has ratio $2k+1$, as we show in the following Proposition, showing that finding protosets with prescribed ratio is much simpler in one dimension.

However, our intuition is that the bigger the disbalance of the protoset, the smaller the upper bound on its Heesch number is. This claim, while in need of a proof, is intuitively justified by observing that in highly disbalanced protosets, the disbalance of the patches in the other direction is of the same order. And the bigger the disbalance is, the bigger the difference is between the eventually required ratio and the one observed in a patch. Our intuition is therefore that this discrepancy will result in the impossibility of formation of larger coronas sooner, the bigger the disbalance. Following this, we should search for protosets that have ratio $q \approx 1$, but we postpone the study of this direction for the next version of this paper.

We just give a description of the protoset with ratio $2k+1$ but postpone the proof (which is quite straightforward once you draw the tiles) of this fact for the next version. Tiles will be described in terms of summands rather than precise numbers to potentially make the drawing of it easier for the brave reader. In what follows, $\alpha$ is an irrational number smaller than one. 

Let $\mathcal{P}_1 = [0, 0.5] \cup [3, 3.5] \cup [3.5 + (2\alpha +2)\cdot k + \alpha -1, 3.5 + (2\alpha +2)\cdot k + \alpha -1 + 0.5] \cup [3.5 + (2\alpha +2)\cdot k + \alpha -1 + 1, 3.5 + (2\alpha +2)\cdot k + \alpha -1 + 0.5 + 1] \cup [3.5 + (2\alpha +2)\cdot k + \alpha -1 + 0.5 + 1 + 0.5, 3.5 + (2\alpha +2)\cdot k + \alpha -1 + 0.5 + 1 + 0.5 + 2]$ and $\mathcal{P}_2 = [0,0.5] \cup [2, 2+\alpha -1] \cup [2 + \alpha -1 + 0.5, 2 + \alpha -1 + 0.5 + 0.5] \cup [2 + \alpha -1 + 0.5 + 0.5 + 0.5, 2 + \alpha -1 + 0.5 + 0.5 + 0.5 + 0.5]$.

\begin{prop}
 The protoset $\mathscr{P} = \left \{\mathcal{P}_1, \mathcal{P}_2 \right \},$ has ratio $2k+1.$
\end{prop}

We also wonder if protosets with arbitrary ratio $q$ can be constructed, but again postpone this for future versions.

\section{Computational results}

In this section we will showcase some of the computational results done by searching the space of all square--based notched prisms up to certain height.

All the tiles were found using a SAT solver which tried to satisfy some set of conditions which exactly explained the behaviour of having Heesch number $k.$ Since searching through the whole space of tiles of a certain height would be too computationally demanding, we focused on specific families of tiles which were easy to describe to the computer. Also, while it is expected that hexagonal--based prisms would have bigger Heesch numbers at the same height (as is the case in the height 1 case), running code on them would be too demanding as well, so we restricted on square--based tiles as a proof of concept.

\subsection{Positive results}

We first start by listing positive results and stating under which restrictions they were found. Apart from the figure, a prism of height $h$ will be described as an ordered $h$-tuple of quadruples whose elements are in $\left \{  -1,0,1\right \},$ which we call its \textit{canonical description}. The way encoding works is that we fix a particular side of the prism (one of four) and list (from the bottom-most level up) the indentations on the sides of each level in the counterclockwise order -- $0$ representing a straight face, and $-1$ and $1$ representing faces with nicks and bumps, respectively, c.f. Figure \ref{fig:p1}. To make the notation less cumbersome, we will abbreviate several consecutive layers that are marked in the same way by putting the description of that layer to the adequate power. We also abbreviate by $\mathbf{H = k}$ the fact that a particular prism has Heesch number $k,$ and by $\mathbf{h = l}$ we denote that it has height $l.$ Since all the prisms are imbalanced, the absolute value of the difference in the number of their bumps and nicks will be denoted by $d.$ Each prism will be given, together with its canonical representation, with a figure showing it. The left part of each figure shows two sides of the prism visible from a fixed point of view, while the second shows the other two sides from the same viewpoint -- in those figures the first two sides are made invisible, so the reader can easily imagine putting one figure (left) on top of the other (right).

All the positive examples have Heesch number $2,$ and while this doesn't seem like much, we observe that the number of tiles in 2 full coronas around our prisms is around $90$, which is comparable to the number of tiles present in all $6$ coronas around the Heesch number $6$ example from Figure \ref{hes6} -- around 170.

We also show how these tiles can be generalized to give 6 infinite classes of tiles with finite Heesch numbers that are at least 2. We begin by showing the positive results.

\paragraph*{$\mathcal{P}_1$: H=2, h=4.}
Tile description:
\[
((-1,-1,1,1), (-1,-1,-1,1), (1,-1,-1,1)^2).
\]

This tile has $d=2.$

\paragraph*{$\mathcal{P}_2$:H=2, h=4.}
Tile description:
\[
((-1,-1,0,1), (0,-1,-1,1), (1,-1,-1,1)^2).
\]

This tile was found while searching for a prism with at least one flat side.

\begin{figure}[!htb]
  \centering
  \begin{minipage}[t]{.48\linewidth}
    \centering
  \begingroup
% Input-ready figure for p1; include inside a figure environment.
\definecolor{pigment}{HTML}{53677A}
\colorlet{bodyLeft}{pigment!52}
\colorlet{bodyRight}{pigment!36}
\colorlet{bodyTop}{pigment!44}
\colorlet{bumpA}{pigment!72}
\colorlet{bumpB}{pigment!62}
\colorlet{bumpC}{pigment!45}
\colorlet{bumpD}{pigment!68}
\colorlet{nickA}{pigment!27}
\colorlet{nickB}{pigment!39}
\colorlet{nickC}{pigment!17}
\colorlet{nickD}{pigment!25}
\colorlet{floorFar}{pigment!65}
\colorlet{floorNear}{pigment!28}
\tikzset{edge/.style={draw=black!80,line width=.45pt,line join=round},
  hidden/.style={draw=black!43,line width=.30pt,
    dash pattern=on 1.7pt off 1.5pt,line join=round},
  ghost/.style={draw=black!42,line width=.32pt,
    dash pattern=on .7pt off 1.45pt,line join=round}}
% Square-pyramidal base .30 by .30, apex depth .09.
% Feature on x=0, viewed from negative x and negative y.
\newcommand{\Xfeature}[2]{%
  \pgfmathsetmacro{\lo}{#1+.35}
  \pgfmathsetmacro{\mid}{#1+.5}
  \pgfmathsetmacro{\hi}{#1+.65}
  \pgfmathsetmacro{\dep}{-.09*(#2)}
  \ifnum#2=1
    \path[edge,fill=bumpA] (0,.35,\lo)--(0,.65,\lo)--(\dep,.5,\mid)--cycle;
    \path[edge,fill=bumpB] (0,.65,\lo)--(0,.65,\hi)--(\dep,.5,\mid)--cycle;
    \path[edge,fill=bumpC] (0,.65,\hi)--(0,.35,\hi)--(\dep,.5,\mid)--cycle;
    \path[edge,fill=bumpD] (0,.35,\hi)--(0,.35,\lo)--(\dep,.5,\mid)--cycle;
  \else
    \pgfmathsetmacro{\lo}{#1+.35}
    \pgfmathsetmacro{\hi}{#1+.65}
    \def\dep{.09}
    \begin{scope}
      \clip (0,.35,\lo)--(0,.65,\lo)--(0,.65,\hi)--(0,.35,\hi)--cycle;
      \path[fill=nickA] (0,.35,\lo)--(0,.65,\lo)--(\dep,.5,\mid)--cycle;
      \path[fill=nickB] (0,.65,\lo)--(0,.65,\hi)--(\dep,.5,\mid)--cycle;
      \path[fill=nickC] (0,.65,\hi)--(0,.35,\hi)--(\dep,.5,\mid)--cycle;
      \path[fill=nickD] (0,.35,\hi)--(0,.35,\lo)--(\dep,.5,\mid)--cycle;
      \draw[black!65,line width=.3pt] (\dep,.5,\mid)--(0,.35,\lo)
        (\dep,.5,\mid)--(0,.65,\lo)
        (\dep,.5,\mid)--(0,.65,\hi)
        (\dep,.5,\mid)--(0,.35,\hi);
    \end{scope}
    \draw[edge] (0,.35,\lo)--(0,.65,\lo)--(0,.65,\hi)--(0,.35,\hi)--cycle;
  \fi
}
% Feature on y=0.
\newcommand{\Yfeature}[2]{%
  \pgfmathsetmacro{\lo}{#1+.35}
  \pgfmathsetmacro{\mid}{#1+.5}
  \pgfmathsetmacro{\hi}{#1+.65}
  \pgfmathsetmacro{\dep}{-.09*(#2)}
  \ifnum#2=1
    \path[edge,fill=bumpA] (.35,0,\lo)--(.65,0,\lo)--(.5,\dep,\mid)--cycle;
    \path[edge,fill=bumpB] (.65,0,\lo)--(.65,0,\hi)--(.5,\dep,\mid)--cycle;
    \path[edge,fill=bumpC] (.65,0,\hi)--(.35,0,\hi)--(.5,\dep,\mid)--cycle;
    \path[edge,fill=bumpD] (.35,0,\hi)--(.35,0,\lo)--(.5,\dep,\mid)--cycle;
  \else
    \pgfmathsetmacro{\lo}{#1+.35}
    \pgfmathsetmacro{\hi}{#1+.65}
    \def\dep{.09}
    \begin{scope}
      \clip (.35,0,\lo)--(.65,0,\lo)--(.65,0,\hi)--(.35,0,\hi)--cycle;
      \path[fill=nickA] (.35,0,\lo)--(.65,0,\lo)--(.5,\dep,\mid)--cycle;
      \path[fill=nickB] (.65,0,\lo)--(.65,0,\hi)--(.5,\dep,\mid)--cycle;
      \path[fill=nickC] (.65,0,\hi)--(.35,0,\hi)--(.5,\dep,\mid)--cycle;
      \path[fill=nickD] (.35,0,\hi)--(.35,0,\lo)--(.5,\dep,\mid)--cycle;
      \draw[black!65,line width=.3pt] (.5,\dep,\mid)--(.35,0,\lo)
        (.5,\dep,\mid)--(.65,0,\lo)
        (.5,\dep,\mid)--(.65,0,\hi)
        (.5,\dep,\mid)--(.35,0,\hi);
    \end{scope}
    \draw[edge] (.35,0,\lo)--(.65,0,\lo)--(.65,0,\hi)--(.35,0,\hi)--cycle;
  \fi
}

\newcommand{\prismbody}{%
  \path[edge,fill=bodyLeft] (0,0,0)--(0,1,0)--(0,1,4)--(0,0,4)--cycle;
  \path[edge,fill=bodyRight] (0,0,0)--(1,0,0)--(1,0,4)--(0,0,4)--cycle;
  \path[edge,fill=bodyTop] (0,0,4)--(1,0,4)--(1,1,4)--(0,1,4)--cycle;
  \foreach \level in {1,2,3}{%
    \draw[hidden] (0,1,\level)--(0,0,\level)--(1,0,\level);
  }
}

\resizebox{.715\linewidth}{!}{%
\begin{tikzpicture}[font=\small]
  \begin{scope}[x={(1.08253cm,.625cm)},y={(-1.08253cm,.625cm)},z={(0cm,1.25cm)}]
    \prismbody
    \foreach \j/\sgn in {0/1,1/1,2/1,3/1}{\Xfeature{\j}{\sgn}}
    \foreach \j/\sgn in {0/-1,1/-1,2/1,3/1}{\Yfeature{\j}{\sgn}}
  \end{scope}
  \begin{scope}[shift={(4.2cm,0)},
      x={(1.08253cm,.625cm)},y={(-1.08253cm,.625cm)},z={(0cm,1.25cm)}]
    % The bottom face is the floor of the open shell.
    \shade[draw=black!70,line width=.45pt,top color=floorFar,bottom color=floorNear] (0,0,0)--(1,0,0)--(1,1,0)--(0,1,0)--cycle;
    % Actual rear faces: x=1 and y=1.
    \path[edge,fill=bodyRight] (1,0,0)--(1,1,0)--(1,1,4)--(1,0,4)--cycle;
    \path[edge,fill=bodyLeft] (0,1,0)--(1,1,0)--(1,1,4)--(0,1,4)--cycle;
    \path[fill=bodyTop,opacity=.3]
      (0,0,4)--(1,0,4)--(1,1,4)--(0,1,4)--cycle;
    \foreach \level in {1,2,3}{%
      \draw[hidden] (1,0,\level)--(1,1,\level)--(0,1,\level);
    }
    % Mirrored coordinates draw features on the physical rear faces.
    \begin{scope}[shift={(1,0,0)},x={(-1.08253cm,-.625cm)}]
      \foreach \j/\sgn in {0/-1,1/-1,2/-1,3/-1}{\Xfeature{\j}{\sgn}}
    \end{scope}
    \begin{scope}[shift={(0,1,0)},y={(1.08253cm,-.625cm)}]
      \foreach \j/\sgn in {0/1,1/-1,2/-1,3/-1}{\Yfeature{\j}{\sgn}}
    \end{scope}
    % The dotted outline records the removed front walls.
    \draw[ghost]
      (0,0,0)--(0,0,4)--(0,1,4)--(0,1,0)--cycle
      (0,0,0)--(1,0,0)--(1,0,4)--(0,0,4)
      (0,1,4)--(1,1,4)--(1,0,4)
      (1,0,0)--(1,1,0)--(0,1,0)
      (1,1,0)--(1,1,4);

  \end{scope}
  \node at (2.1,6.75) {$\mathcal P_1$};
  \node at (0,-.47) {(a) front};
  \node at (4.2,-.47) {(b) rear faces};
\end{tikzpicture}%
}%
\endgroup
  \caption{Two views of $\mathcal{P}_1$.}
  \label{fig:p1}
  \end{minipage}%
  \hfill
  \begin{minipage}[t]{.48\linewidth}
    \centering
  \begingroup
% Input-ready figure for p2; include inside a figure environment.
\definecolor{pigment}{HTML}{53677A}
\colorlet{bodyLeft}{pigment!52}
\colorlet{bodyRight}{pigment!36}
\colorlet{bodyTop}{pigment!44}
\colorlet{bumpA}{pigment!72}
\colorlet{bumpB}{pigment!62}
\colorlet{bumpC}{pigment!45}
\colorlet{bumpD}{pigment!68}
\colorlet{nickA}{pigment!27}
\colorlet{nickB}{pigment!39}
\colorlet{nickC}{pigment!17}
\colorlet{nickD}{pigment!25}
\colorlet{floorFar}{pigment!65}
\colorlet{floorNear}{pigment!28}
\tikzset{edge/.style={draw=black!80,line width=.45pt,line join=round},
  hidden/.style={draw=black!43,line width=.30pt,
    dash pattern=on 1.7pt off 1.5pt,line join=round},
  ghost/.style={draw=black!42,line width=.32pt,
    dash pattern=on .7pt off 1.45pt,line join=round}}
% Square-pyramidal base .30 by .30, apex depth .09.
% Feature on x=0, viewed from negative x and negative y.
\newcommand{\Xfeature}[2]{%
  \pgfmathsetmacro{\lo}{#1+.35}
  \pgfmathsetmacro{\mid}{#1+.5}
  \pgfmathsetmacro{\hi}{#1+.65}
  \pgfmathsetmacro{\dep}{-.09*(#2)}
  \ifnum#2=1
    \path[edge,fill=bumpA] (0,.35,\lo)--(0,.65,\lo)--(\dep,.5,\mid)--cycle;
    \path[edge,fill=bumpB] (0,.65,\lo)--(0,.65,\hi)--(\dep,.5,\mid)--cycle;
    \path[edge,fill=bumpC] (0,.65,\hi)--(0,.35,\hi)--(\dep,.5,\mid)--cycle;
    \path[edge,fill=bumpD] (0,.35,\hi)--(0,.35,\lo)--(\dep,.5,\mid)--cycle;
  \else
    \pgfmathsetmacro{\lo}{#1+.35}
    \pgfmathsetmacro{\hi}{#1+.65}
    \def\dep{.09}
    \begin{scope}
      \clip (0,.35,\lo)--(0,.65,\lo)--(0,.65,\hi)--(0,.35,\hi)--cycle;
      \path[fill=nickA] (0,.35,\lo)--(0,.65,\lo)--(\dep,.5,\mid)--cycle;
      \path[fill=nickB] (0,.65,\lo)--(0,.65,\hi)--(\dep,.5,\mid)--cycle;
      \path[fill=nickC] (0,.65,\hi)--(0,.35,\hi)--(\dep,.5,\mid)--cycle;
      \path[fill=nickD] (0,.35,\hi)--(0,.35,\lo)--(\dep,.5,\mid)--cycle;
      \draw[black!65,line width=.3pt] (\dep,.5,\mid)--(0,.35,\lo)
        (\dep,.5,\mid)--(0,.65,\lo)
        (\dep,.5,\mid)--(0,.65,\hi)
        (\dep,.5,\mid)--(0,.35,\hi);
    \end{scope}
    \draw[edge] (0,.35,\lo)--(0,.65,\lo)--(0,.65,\hi)--(0,.35,\hi)--cycle;
  \fi
}
% Feature on y=0.
\newcommand{\Yfeature}[2]{%
  \pgfmathsetmacro{\lo}{#1+.35}
  \pgfmathsetmacro{\mid}{#1+.5}
  \pgfmathsetmacro{\hi}{#1+.65}
  \pgfmathsetmacro{\dep}{-.09*(#2)}
  \ifnum#2=1
    \path[edge,fill=bumpA] (.35,0,\lo)--(.65,0,\lo)--(.5,\dep,\mid)--cycle;
    \path[edge,fill=bumpB] (.65,0,\lo)--(.65,0,\hi)--(.5,\dep,\mid)--cycle;
    \path[edge,fill=bumpC] (.65,0,\hi)--(.35,0,\hi)--(.5,\dep,\mid)--cycle;
    \path[edge,fill=bumpD] (.35,0,\hi)--(.35,0,\lo)--(.5,\dep,\mid)--cycle;
  \else
    \pgfmathsetmacro{\lo}{#1+.35}
    \pgfmathsetmacro{\hi}{#1+.65}
    \def\dep{.09}
    \begin{scope}
      \clip (.35,0,\lo)--(.65,0,\lo)--(.65,0,\hi)--(.35,0,\hi)--cycle;
      \path[fill=nickA] (.35,0,\lo)--(.65,0,\lo)--(.5,\dep,\mid)--cycle;
      \path[fill=nickB] (.65,0,\lo)--(.65,0,\hi)--(.5,\dep,\mid)--cycle;
      \path[fill=nickC] (.65,0,\hi)--(.35,0,\hi)--(.5,\dep,\mid)--cycle;
      \path[fill=nickD] (.35,0,\hi)--(.35,0,\lo)--(.5,\dep,\mid)--cycle;
      \draw[black!65,line width=.3pt] (.5,\dep,\mid)--(.35,0,\lo)
        (.5,\dep,\mid)--(.65,0,\lo)
        (.5,\dep,\mid)--(.65,0,\hi)
        (.5,\dep,\mid)--(.35,0,\hi);
    \end{scope}
    \draw[edge] (.35,0,\lo)--(.65,0,\lo)--(.65,0,\hi)--(.35,0,\hi)--cycle;
  \fi
}

\newcommand{\prismbody}{%
  \path[edge,fill=bodyLeft] (0,0,0)--(0,1,0)--(0,1,4)--(0,0,4)--cycle;
  \path[edge,fill=bodyRight] (0,0,0)--(1,0,0)--(1,0,4)--(0,0,4)--cycle;
  \path[edge,fill=bodyTop] (0,0,4)--(1,0,4)--(1,1,4)--(0,1,4)--cycle;
  \foreach \level in {1,2,3}{%
    \draw[hidden] (0,1,\level)--(0,0,\level)--(1,0,\level);
  }
}

\resizebox{.715\linewidth}{!}{%
\begin{tikzpicture}[font=\small]
  \begin{scope}[x={(1.08253cm,.625cm)},y={(-1.08253cm,.625cm)},z={(0cm,1.25cm)}]
    \prismbody
    \foreach \j/\sgn in {0/1,1/1,2/1,3/1}{\Xfeature{\j}{\sgn}}
    \foreach \j/\sgn in {0/-1,2/1,3/1}{\Yfeature{\j}{\sgn}}
  \end{scope}
  \begin{scope}[shift={(4.2cm,0)},
      x={(1.08253cm,.625cm)},y={(-1.08253cm,.625cm)},z={(0cm,1.25cm)}]
    % The bottom face is the floor of the open shell.
    \shade[draw=black!70,line width=.45pt,top color=floorFar,bottom color=floorNear] (0,0,0)--(1,0,0)--(1,1,0)--(0,1,0)--cycle;
    % Actual rear faces: x=1 and y=1.
    \path[edge,fill=bodyRight] (1,0,0)--(1,1,0)--(1,1,4)--(1,0,4)--cycle;
    \path[edge,fill=bodyLeft] (0,1,0)--(1,1,0)--(1,1,4)--(0,1,4)--cycle;
    \path[fill=bodyTop,opacity=.3]
      (0,0,4)--(1,0,4)--(1,1,4)--(0,1,4)--cycle;
    \foreach \level in {1,2,3}{%
      \draw[hidden] (1,0,\level)--(1,1,\level)--(0,1,\level);
    }
    % Mirrored coordinates draw features on the physical rear faces.
    \begin{scope}[shift={(1,0,0)},x={(-1.08253cm,-.625cm)}]
      \foreach \j/\sgn in {0/-1,1/-1,2/-1,3/-1}{\Xfeature{\j}{\sgn}}
    \end{scope}
    \begin{scope}[shift={(0,1,0)},y={(1.08253cm,-.625cm)}]
      \foreach \j/\sgn in {1/-1,2/-1,3/-1}{\Yfeature{\j}{\sgn}}
    \end{scope}
    % The dotted outline records the removed front walls.
    \draw[ghost]
      (0,0,0)--(0,0,4)--(0,1,4)--(0,1,0)--cycle
      (0,0,0)--(1,0,0)--(1,0,4)--(0,0,4)
      (0,1,4)--(1,1,4)--(1,0,4)
      (1,0,0)--(1,1,0)--(0,1,0)
      (1,1,0)--(1,1,4);

  \end{scope}
  \node at (2.1,6.75) {$\mathcal P_2$};
  \node at (0,-.47) {(a) front};
  \node at (4.2,-.47) {(b) rear faces};
\end{tikzpicture}%
}%
\endgroup
  \caption{Two views of $\mathcal{P}_2$.}
  \label{fig:p2}
  \end{minipage}
\end{figure}

\paragraph*{$\mathcal{P}_3$: H=2, h=8.}
Tile description:
\[
((-1,-1,1,1), (1,-1,1,1)^2, (1,-1,-1,1)^5).
\]

This tile was found while searching for a prism with no flat faces.

This tile has $d=4.$

\paragraph*{$\mathcal{P}_4$: H=2, h=8.}
Tile description:
\[
\begin{aligned}
&((-1,-1,-1,-1), (-1,0,-1,-1), (-1,1,-1,-1), (0,1,-1,-1),\\
&\qquad (1,1,-1,-1), (1,1,-1,0), (1,1,-1,1), (1,1,0,1)).
\end{aligned}
\]

This example is particularly interesting since it exhibits a "spiral structure" -- the flat faces move in a spiral. Also, each side is "increasing", in the sense that each side consists of some faces with a nick, then a flat face, then some faces with a bump. It is also interesting that in one of the surroundings that achieve Heesch number $2$, prisms can be paired so that each is in a block of two that are stacked on top of each other, where all pairs have same choice of an upper prism being rotated around the lower prism.

This tile has $d=4.$

\begin{figure}[!htb]
  \centering
  \begin{minipage}[t]{.48\linewidth}
    \centering
  \begingroup
% Input-ready figure for p3; include inside a figure environment.
\definecolor{pigment}{HTML}{53677A}
\colorlet{bodyLeft}{pigment!52}
\colorlet{bodyRight}{pigment!36}
\colorlet{bodyTop}{pigment!44}
\colorlet{bumpA}{pigment!72}
\colorlet{bumpB}{pigment!62}
\colorlet{bumpC}{pigment!45}
\colorlet{bumpD}{pigment!68}
\colorlet{nickA}{pigment!27}
\colorlet{nickB}{pigment!39}
\colorlet{nickC}{pigment!17}
\colorlet{nickD}{pigment!25}
\colorlet{floorFar}{pigment!65}
\colorlet{floorNear}{pigment!28}
\tikzset{edge/.style={draw=black!80,line width=.45pt,line join=round},
  hidden/.style={draw=black!43,line width=.30pt,
    dash pattern=on 1.7pt off 1.5pt,line join=round},
  ghost/.style={draw=black!42,line width=.32pt,
    dash pattern=on .7pt off 1.45pt,line join=round}}
% Square-pyramidal base .30 by .30, apex depth .09.
% Feature on x=0, viewed from negative x and negative y.
\newcommand{\Xfeature}[2]{%
  \pgfmathsetmacro{\lo}{#1+.35}
  \pgfmathsetmacro{\mid}{#1+.5}
  \pgfmathsetmacro{\hi}{#1+.65}
  \pgfmathsetmacro{\dep}{-.09*(#2)}
  \ifnum#2=1
    \path[edge,fill=bumpA] (0,.35,\lo)--(0,.65,\lo)--(\dep,.5,\mid)--cycle;
    \path[edge,fill=bumpB] (0,.65,\lo)--(0,.65,\hi)--(\dep,.5,\mid)--cycle;
    \path[edge,fill=bumpC] (0,.65,\hi)--(0,.35,\hi)--(\dep,.5,\mid)--cycle;
    \path[edge,fill=bumpD] (0,.35,\hi)--(0,.35,\lo)--(\dep,.5,\mid)--cycle;
  \else
    \pgfmathsetmacro{\lo}{#1+.35}
    \pgfmathsetmacro{\hi}{#1+.65}
    \def\dep{.09}
    \begin{scope}
      \clip (0,.35,\lo)--(0,.65,\lo)--(0,.65,\hi)--(0,.35,\hi)--cycle;
      \path[fill=nickA] (0,.35,\lo)--(0,.65,\lo)--(\dep,.5,\mid)--cycle;
      \path[fill=nickB] (0,.65,\lo)--(0,.65,\hi)--(\dep,.5,\mid)--cycle;
      \path[fill=nickC] (0,.65,\hi)--(0,.35,\hi)--(\dep,.5,\mid)--cycle;
      \path[fill=nickD] (0,.35,\hi)--(0,.35,\lo)--(\dep,.5,\mid)--cycle;
      \draw[black!65,line width=.3pt] (\dep,.5,\mid)--(0,.35,\lo)
        (\dep,.5,\mid)--(0,.65,\lo)
        (\dep,.5,\mid)--(0,.65,\hi)
        (\dep,.5,\mid)--(0,.35,\hi);
    \end{scope}
    \draw[edge] (0,.35,\lo)--(0,.65,\lo)--(0,.65,\hi)--(0,.35,\hi)--cycle;
  \fi
}
% Feature on y=0.
\newcommand{\Yfeature}[2]{%
  \pgfmathsetmacro{\lo}{#1+.35}
  \pgfmathsetmacro{\mid}{#1+.5}
  \pgfmathsetmacro{\hi}{#1+.65}
  \pgfmathsetmacro{\dep}{-.09*(#2)}
  \ifnum#2=1
    \path[edge,fill=bumpA] (.35,0,\lo)--(.65,0,\lo)--(.5,\dep,\mid)--cycle;
    \path[edge,fill=bumpB] (.65,0,\lo)--(.65,0,\hi)--(.5,\dep,\mid)--cycle;
    \path[edge,fill=bumpC] (.65,0,\hi)--(.35,0,\hi)--(.5,\dep,\mid)--cycle;
    \path[edge,fill=bumpD] (.35,0,\hi)--(.35,0,\lo)--(.5,\dep,\mid)--cycle;
  \else
    \pgfmathsetmacro{\lo}{#1+.35}
    \pgfmathsetmacro{\hi}{#1+.65}
    \def\dep{.09}
    \begin{scope}
      \clip (.35,0,\lo)--(.65,0,\lo)--(.65,0,\hi)--(.35,0,\hi)--cycle;
      \path[fill=nickA] (.35,0,\lo)--(.65,0,\lo)--(.5,\dep,\mid)--cycle;
      \path[fill=nickB] (.65,0,\lo)--(.65,0,\hi)--(.5,\dep,\mid)--cycle;
      \path[fill=nickC] (.65,0,\hi)--(.35,0,\hi)--(.5,\dep,\mid)--cycle;
      \path[fill=nickD] (.35,0,\hi)--(.35,0,\lo)--(.5,\dep,\mid)--cycle;
      \draw[black!65,line width=.3pt] (.5,\dep,\mid)--(.35,0,\lo)
        (.5,\dep,\mid)--(.65,0,\lo)
        (.5,\dep,\mid)--(.65,0,\hi)
        (.5,\dep,\mid)--(.35,0,\hi);
    \end{scope}
    \draw[edge] (.35,0,\lo)--(.65,0,\lo)--(.65,0,\hi)--(.35,0,\hi)--cycle;
  \fi
}

\newcommand{\prismbody}{%
  \path[edge,fill=bodyLeft] (0,0,0)--(0,1,0)--(0,1,8)--(0,0,8)--cycle;
  \path[edge,fill=bodyRight] (0,0,0)--(1,0,0)--(1,0,8)--(0,0,8)--cycle;
  \path[edge,fill=bodyTop] (0,0,8)--(1,0,8)--(1,1,8)--(0,1,8)--cycle;
  \foreach \level in {1,2,3,4,5,6,7}{%
    \draw[hidden] (0,1,\level)--(0,0,\level)--(1,0,\level);
  }
}

\resizebox{.715\linewidth}{!}{%
\begin{tikzpicture}[font=\small]
  \begin{scope}[x={(1.08253cm,.625cm)},y={(-1.08253cm,.625cm)},z={(0cm,1.25cm)}]
    \prismbody
    \foreach \j/\sgn in {0/1,1/1,2/1,3/1,4/1,5/1,6/1,7/1}{\Xfeature{\j}{\sgn}}
    \foreach \j/\sgn in {0/-1,1/1,2/1,3/1,4/1,5/1,6/1,7/1}{\Yfeature{\j}{\sgn}}
  \end{scope}
  \begin{scope}[shift={(4.2cm,0)},
      x={(1.08253cm,.625cm)},y={(-1.08253cm,.625cm)},z={(0cm,1.25cm)}]
    % The bottom face is the floor of the open shell.
    \shade[draw=black!70,line width=.45pt,top color=floorFar,bottom color=floorNear] (0,0,0)--(1,0,0)--(1,1,0)--(0,1,0)--cycle;
    % Actual rear faces: x=1 and y=1.
    \path[edge,fill=bodyRight] (1,0,0)--(1,1,0)--(1,1,8)--(1,0,8)--cycle;
    \path[edge,fill=bodyLeft] (0,1,0)--(1,1,0)--(1,1,8)--(0,1,8)--cycle;
    \path[fill=bodyTop,opacity=.3]
      (0,0,8)--(1,0,8)--(1,1,8)--(0,1,8)--cycle;
    \foreach \level in {1,2,3,4,5,6,7}{%
      \draw[hidden] (1,0,\level)--(1,1,\level)--(0,1,\level);
    }
    % Mirrored coordinates draw features on the physical rear faces.
    \begin{scope}[shift={(1,0,0)},x={(-1.08253cm,-.625cm)}]
      \foreach \j/\sgn in {0/-1,1/-1,2/-1,3/-1,4/-1,5/-1,6/-1,7/-1}{\Xfeature{\j}{\sgn}}
    \end{scope}
    \begin{scope}[shift={(0,1,0)},y={(1.08253cm,-.625cm)}]
      \foreach \j/\sgn in {0/1,1/1,2/1,3/-1,4/-1,5/-1,6/-1,7/-1}{\Yfeature{\j}{\sgn}}
    \end{scope}
    % The dotted outline records the removed front walls.
    \draw[ghost]
      (0,0,0)--(0,0,8)--(0,1,8)--(0,1,0)--cycle
      (0,0,0)--(1,0,0)--(1,0,8)--(0,0,8)
      (0,1,8)--(1,1,8)--(1,0,8)
      (1,0,0)--(1,1,0)--(0,1,0)
      (1,1,0)--(1,1,8);

  \end{scope}
  \node at (2.1,11.75) {$\mathcal P_3$};
  \node at (0,-.47) {(a) front};
  \node at (4.2,-.47) {(b) rear faces};
\end{tikzpicture}%
}%
\endgroup
  \caption{Two views of $\mathcal{P}_3$.}
  \label{fig:p3}
  \end{minipage}%
  \hfill
  \begin{minipage}[t]{.48\linewidth}
    \centering
  \begingroup
% Input-ready figure for p4; include inside a figure environment.
\definecolor{pigment}{HTML}{53677A}
\colorlet{bodyLeft}{pigment!52}
\colorlet{bodyRight}{pigment!36}
\colorlet{bodyTop}{pigment!44}
\colorlet{bumpA}{pigment!72}
\colorlet{bumpB}{pigment!62}
\colorlet{bumpC}{pigment!45}
\colorlet{bumpD}{pigment!68}
\colorlet{nickA}{pigment!27}
\colorlet{nickB}{pigment!39}
\colorlet{nickC}{pigment!17}
\colorlet{nickD}{pigment!25}
\colorlet{floorFar}{pigment!65}
\colorlet{floorNear}{pigment!28}
\tikzset{edge/.style={draw=black!80,line width=.45pt,line join=round},
  hidden/.style={draw=black!43,line width=.30pt,
    dash pattern=on 1.7pt off 1.5pt,line join=round},
  ghost/.style={draw=black!42,line width=.32pt,
    dash pattern=on .7pt off 1.45pt,line join=round}}
% Square-pyramidal base .30 by .30, apex depth .09.
% Feature on x=0, viewed from negative x and negative y.
\newcommand{\Xfeature}[2]{%
  \pgfmathsetmacro{\lo}{#1+.35}
  \pgfmathsetmacro{\mid}{#1+.5}
  \pgfmathsetmacro{\hi}{#1+.65}
  \pgfmathsetmacro{\dep}{-.09*(#2)}
  \ifnum#2=1
    \path[edge,fill=bumpA] (0,.35,\lo)--(0,.65,\lo)--(\dep,.5,\mid)--cycle;
    \path[edge,fill=bumpB] (0,.65,\lo)--(0,.65,\hi)--(\dep,.5,\mid)--cycle;
    \path[edge,fill=bumpC] (0,.65,\hi)--(0,.35,\hi)--(\dep,.5,\mid)--cycle;
    \path[edge,fill=bumpD] (0,.35,\hi)--(0,.35,\lo)--(\dep,.5,\mid)--cycle;
  \else
    \pgfmathsetmacro{\lo}{#1+.35}
    \pgfmathsetmacro{\hi}{#1+.65}
    \def\dep{.09}
    \begin{scope}
      \clip (0,.35,\lo)--(0,.65,\lo)--(0,.65,\hi)--(0,.35,\hi)--cycle;
      \path[fill=nickA] (0,.35,\lo)--(0,.65,\lo)--(\dep,.5,\mid)--cycle;
      \path[fill=nickB] (0,.65,\lo)--(0,.65,\hi)--(\dep,.5,\mid)--cycle;
      \path[fill=nickC] (0,.65,\hi)--(0,.35,\hi)--(\dep,.5,\mid)--cycle;
      \path[fill=nickD] (0,.35,\hi)--(0,.35,\lo)--(\dep,.5,\mid)--cycle;
      \draw[black!65,line width=.3pt] (\dep,.5,\mid)--(0,.35,\lo)
        (\dep,.5,\mid)--(0,.65,\lo)
        (\dep,.5,\mid)--(0,.65,\hi)
        (\dep,.5,\mid)--(0,.35,\hi);
    \end{scope}
    \draw[edge] (0,.35,\lo)--(0,.65,\lo)--(0,.65,\hi)--(0,.35,\hi)--cycle;
  \fi
}
% Feature on y=0.
\newcommand{\Yfeature}[2]{%
  \pgfmathsetmacro{\lo}{#1+.35}
  \pgfmathsetmacro{\mid}{#1+.5}
  \pgfmathsetmacro{\hi}{#1+.65}
  \pgfmathsetmacro{\dep}{-.09*(#2)}
  \ifnum#2=1
    \path[edge,fill=bumpA] (.35,0,\lo)--(.65,0,\lo)--(.5,\dep,\mid)--cycle;
    \path[edge,fill=bumpB] (.65,0,\lo)--(.65,0,\hi)--(.5,\dep,\mid)--cycle;
    \path[edge,fill=bumpC] (.65,0,\hi)--(.35,0,\hi)--(.5,\dep,\mid)--cycle;
    \path[edge,fill=bumpD] (.35,0,\hi)--(.35,0,\lo)--(.5,\dep,\mid)--cycle;
  \else
    \pgfmathsetmacro{\lo}{#1+.35}
    \pgfmathsetmacro{\hi}{#1+.65}
    \def\dep{.09}
    \begin{scope}
      \clip (.35,0,\lo)--(.65,0,\lo)--(.65,0,\hi)--(.35,0,\hi)--cycle;
      \path[fill=nickA] (.35,0,\lo)--(.65,0,\lo)--(.5,\dep,\mid)--cycle;
      \path[fill=nickB] (.65,0,\lo)--(.65,0,\hi)--(.5,\dep,\mid)--cycle;
      \path[fill=nickC] (.65,0,\hi)--(.35,0,\hi)--(.5,\dep,\mid)--cycle;
      \path[fill=nickD] (.35,0,\hi)--(.35,0,\lo)--(.5,\dep,\mid)--cycle;
      \draw[black!65,line width=.3pt] (.5,\dep,\mid)--(.35,0,\lo)
        (.5,\dep,\mid)--(.65,0,\lo)
        (.5,\dep,\mid)--(.65,0,\hi)
        (.5,\dep,\mid)--(.35,0,\hi);
    \end{scope}
    \draw[edge] (.35,0,\lo)--(.65,0,\lo)--(.65,0,\hi)--(.35,0,\hi)--cycle;
  \fi
}

\newcommand{\prismbody}{%
  \path[edge,fill=bodyLeft] (0,0,0)--(0,1,0)--(0,1,8)--(0,0,8)--cycle;
  \path[edge,fill=bodyRight] (0,0,0)--(1,0,0)--(1,0,8)--(0,0,8)--cycle;
  \path[edge,fill=bodyTop] (0,0,8)--(1,0,8)--(1,1,8)--(0,1,8)--cycle;
  \foreach \level in {1,2,3,4,5,6,7}{%
    \draw[hidden] (0,1,\level)--(0,0,\level)--(1,0,\level);
  }
}

\resizebox{.715\linewidth}{!}{%
\begin{tikzpicture}[font=\small]
  \begin{scope}[x={(1.08253cm,.625cm)},y={(-1.08253cm,.625cm)},z={(0cm,1.25cm)}]
    \prismbody
    \foreach \j/\sgn in {0/-1,1/-1,2/-1,3/-1,4/-1,6/1,7/1}{\Xfeature{\j}{\sgn}}
    \foreach \j/\sgn in {0/-1,1/-1,2/-1,4/1,5/1,6/1,7/1}{\Yfeature{\j}{\sgn}}
  \end{scope}
  \begin{scope}[shift={(4.2cm,0)},
      x={(1.08253cm,.625cm)},y={(-1.08253cm,.625cm)},z={(0cm,1.25cm)}]
    % The bottom face is the floor of the open shell.
    \shade[draw=black!70,line width=.45pt,top color=floorFar,bottom color=floorNear] (0,0,0)--(1,0,0)--(1,1,0)--(0,1,0)--cycle;
    % Actual rear faces: x=1 and y=1.
    \path[edge,fill=bodyRight] (1,0,0)--(1,1,0)--(1,1,8)--(1,0,8)--cycle;
    \path[edge,fill=bodyLeft] (0,1,0)--(1,1,0)--(1,1,8)--(0,1,8)--cycle;
    \path[fill=bodyTop,opacity=.3]
      (0,0,8)--(1,0,8)--(1,1,8)--(0,1,8)--cycle;
    \foreach \level in {1,2,3,4,5,6,7}{%
      \draw[hidden] (1,0,\level)--(1,1,\level)--(0,1,\level);
    }
    % Mirrored coordinates draw features on the physical rear faces.
    \begin{scope}[shift={(1,0,0)},x={(-1.08253cm,-.625cm)}]
      \foreach \j/\sgn in {0/-1,2/1,3/1,4/1,5/1,6/1,7/1}{\Xfeature{\j}{\sgn}}
    \end{scope}
    \begin{scope}[shift={(0,1,0)},y={(1.08253cm,-.625cm)}]
      \foreach \j/\sgn in {0/-1,1/-1,2/-1,3/-1,4/-1,5/-1,6/-1}{\Yfeature{\j}{\sgn}}
    \end{scope}
    % The dotted outline records the removed front walls.
    \draw[ghost]
      (0,0,0)--(0,0,8)--(0,1,8)--(0,1,0)--cycle
      (0,0,0)--(1,0,0)--(1,0,8)--(0,0,8)
      (0,1,8)--(1,1,8)--(1,0,8)
      (1,0,0)--(1,1,0)--(0,1,0)
      (1,1,0)--(1,1,8);

  \end{scope}
  \node at (2.1,11.75) {$\mathcal P_4$};
  \node at (0,-.47) {(a) front};
  \node at (4.2,-.47) {(b) rear faces};
\end{tikzpicture}%
}%
\endgroup
  \caption{Two views of $\mathcal{P}_4$.}
  \label{fig:p4}
  \end{minipage}
\end{figure}

\paragraph*{$\mathcal{P}_5$: H=2, h=8.}
Tile description:
\[
\begin{aligned}
&((-1,-1,0,0), (0,0,1,1)^2, (-1,-1,0,0), (0,-1,0,0),\\
&\qquad (0,0,1,1), (1,0,0,1), (0,1,0,0)).
\end{aligned}
\]

This tile has $d = 4.$

\paragraph*{$\mathcal{P}_6$: H=2, h=8.}
Tile description:
\[
\begin{aligned}
&((-1,0,0,0), (-1,1,0,0), (0,1,0,0), (0,0,1,0),\\
&\qquad (0,0,1,-1), (0,0,0,-1), (-1,0,0,0), (-1,1,0,0)).
\end{aligned}
\]

This tile has $d = 1.$

\begin{figure}[!htb]
  \centering
  \begin{minipage}[t]{.48\linewidth}
    \centering
  \begingroup
% Input-ready figure for p5; include inside a figure environment.
\definecolor{pigment}{HTML}{53677A}
\colorlet{bodyLeft}{pigment!52}
\colorlet{bodyRight}{pigment!36}
\colorlet{bodyTop}{pigment!44}
\colorlet{bumpA}{pigment!72}
\colorlet{bumpB}{pigment!62}
\colorlet{bumpC}{pigment!45}
\colorlet{bumpD}{pigment!68}
\colorlet{nickA}{pigment!27}
\colorlet{nickB}{pigment!39}
\colorlet{nickC}{pigment!17}
\colorlet{nickD}{pigment!25}
\colorlet{floorFar}{pigment!65}
\colorlet{floorNear}{pigment!28}
\tikzset{edge/.style={draw=black!80,line width=.45pt,line join=round},
  hidden/.style={draw=black!43,line width=.30pt,
    dash pattern=on 1.7pt off 1.5pt,line join=round},
  ghost/.style={draw=black!42,line width=.32pt,
    dash pattern=on .7pt off 1.45pt,line join=round}}
% Square-pyramidal base .30 by .30, apex depth .09.
% Feature on x=0, viewed from negative x and negative y.
\newcommand{\Xfeature}[2]{%
  \pgfmathsetmacro{\lo}{#1+.35}
  \pgfmathsetmacro{\mid}{#1+.5}
  \pgfmathsetmacro{\hi}{#1+.65}
  \pgfmathsetmacro{\dep}{-.09*(#2)}
  \ifnum#2=1
    \path[edge,fill=bumpA] (0,.35,\lo)--(0,.65,\lo)--(\dep,.5,\mid)--cycle;
    \path[edge,fill=bumpB] (0,.65,\lo)--(0,.65,\hi)--(\dep,.5,\mid)--cycle;
    \path[edge,fill=bumpC] (0,.65,\hi)--(0,.35,\hi)--(\dep,.5,\mid)--cycle;
    \path[edge,fill=bumpD] (0,.35,\hi)--(0,.35,\lo)--(\dep,.5,\mid)--cycle;
  \else
    \pgfmathsetmacro{\lo}{#1+.35}
    \pgfmathsetmacro{\hi}{#1+.65}
    \def\dep{.09}
    \begin{scope}
      \clip (0,.35,\lo)--(0,.65,\lo)--(0,.65,\hi)--(0,.35,\hi)--cycle;
      \path[fill=nickA] (0,.35,\lo)--(0,.65,\lo)--(\dep,.5,\mid)--cycle;
      \path[fill=nickB] (0,.65,\lo)--(0,.65,\hi)--(\dep,.5,\mid)--cycle;
      \path[fill=nickC] (0,.65,\hi)--(0,.35,\hi)--(\dep,.5,\mid)--cycle;
      \path[fill=nickD] (0,.35,\hi)--(0,.35,\lo)--(\dep,.5,\mid)--cycle;
      \draw[black!65,line width=.3pt] (\dep,.5,\mid)--(0,.35,\lo)
        (\dep,.5,\mid)--(0,.65,\lo)
        (\dep,.5,\mid)--(0,.65,\hi)
        (\dep,.5,\mid)--(0,.35,\hi);
    \end{scope}
    \draw[edge] (0,.35,\lo)--(0,.65,\lo)--(0,.65,\hi)--(0,.35,\hi)--cycle;
  \fi
}
% Feature on y=0.
\newcommand{\Yfeature}[2]{%
  \pgfmathsetmacro{\lo}{#1+.35}
  \pgfmathsetmacro{\mid}{#1+.5}
  \pgfmathsetmacro{\hi}{#1+.65}
  \pgfmathsetmacro{\dep}{-.09*(#2)}
  \ifnum#2=1
    \path[edge,fill=bumpA] (.35,0,\lo)--(.65,0,\lo)--(.5,\dep,\mid)--cycle;
    \path[edge,fill=bumpB] (.65,0,\lo)--(.65,0,\hi)--(.5,\dep,\mid)--cycle;
    \path[edge,fill=bumpC] (.65,0,\hi)--(.35,0,\hi)--(.5,\dep,\mid)--cycle;
    \path[edge,fill=bumpD] (.35,0,\hi)--(.35,0,\lo)--(.5,\dep,\mid)--cycle;
  \else
    \pgfmathsetmacro{\lo}{#1+.35}
    \pgfmathsetmacro{\hi}{#1+.65}
    \def\dep{.09}
    \begin{scope}
      \clip (.35,0,\lo)--(.65,0,\lo)--(.65,0,\hi)--(.35,0,\hi)--cycle;
      \path[fill=nickA] (.35,0,\lo)--(.65,0,\lo)--(.5,\dep,\mid)--cycle;
      \path[fill=nickB] (.65,0,\lo)--(.65,0,\hi)--(.5,\dep,\mid)--cycle;
      \path[fill=nickC] (.65,0,\hi)--(.35,0,\hi)--(.5,\dep,\mid)--cycle;
      \path[fill=nickD] (.35,0,\hi)--(.35,0,\lo)--(.5,\dep,\mid)--cycle;
      \draw[black!65,line width=.3pt] (.5,\dep,\mid)--(.35,0,\lo)
        (.5,\dep,\mid)--(.65,0,\lo)
        (.5,\dep,\mid)--(.65,0,\hi)
        (.5,\dep,\mid)--(.35,0,\hi);
    \end{scope}
    \draw[edge] (.35,0,\lo)--(.65,0,\lo)--(.65,0,\hi)--(.35,0,\hi)--cycle;
  \fi
}

\newcommand{\prismbody}{%
  \path[edge,fill=bodyLeft] (0,0,0)--(0,1,0)--(0,1,8)--(0,0,8)--cycle;
  \path[edge,fill=bodyRight] (0,0,0)--(1,0,0)--(1,0,8)--(0,0,8)--cycle;
  \path[edge,fill=bodyTop] (0,0,8)--(1,0,8)--(1,1,8)--(0,1,8)--cycle;
  \foreach \level in {1,2,3,4,5,6,7}{%
    \draw[hidden] (0,1,\level)--(0,0,\level)--(1,0,\level);
  }
}

\resizebox{.715\linewidth}{!}{%
\begin{tikzpicture}[font=\small]
  \begin{scope}[x={(1.08253cm,.625cm)},y={(-1.08253cm,.625cm)},z={(0cm,1.25cm)}]
    \prismbody
    \foreach \j/\sgn in {1/1,2/1,5/1,6/1}{\Xfeature{\j}{\sgn}}
    \foreach \j/\sgn in {0/-1,3/-1,6/1}{\Yfeature{\j}{\sgn}}
  \end{scope}
  \begin{scope}[shift={(4.2cm,0)},
      x={(1.08253cm,.625cm)},y={(-1.08253cm,.625cm)},z={(0cm,1.25cm)}]
    % The bottom face is the floor of the open shell.
    \shade[draw=black!70,line width=.45pt,top color=floorFar,bottom color=floorNear] (0,0,0)--(1,0,0)--(1,1,0)--(0,1,0)--cycle;
    % Actual rear faces: x=1 and y=1.
    \path[edge,fill=bodyRight] (1,0,0)--(1,1,0)--(1,1,8)--(1,0,8)--cycle;
    \path[edge,fill=bodyLeft] (0,1,0)--(1,1,0)--(1,1,8)--(0,1,8)--cycle;
    \path[fill=bodyTop,opacity=.3]
      (0,0,8)--(1,0,8)--(1,1,8)--(0,1,8)--cycle;
    \foreach \level in {1,2,3,4,5,6,7}{%
      \draw[hidden] (1,0,\level)--(1,1,\level)--(0,1,\level);
    }
    % Mirrored coordinates draw features on the physical rear faces.
    \begin{scope}[shift={(1,0,0)},x={(-1.08253cm,-.625cm)}]
      \foreach \j/\sgn in {0/-1,3/-1,4/-1,7/1}{\Xfeature{\j}{\sgn}}
    \end{scope}
    \begin{scope}[shift={(0,1,0)},y={(1.08253cm,-.625cm)}]
      \foreach \j/\sgn in {1/1,2/1,5/1}{\Yfeature{\j}{\sgn}}
    \end{scope}
    % The dotted outline records the removed front walls.
    \draw[ghost]
      (0,0,0)--(0,0,8)--(0,1,8)--(0,1,0)--cycle
      (0,0,0)--(1,0,0)--(1,0,8)--(0,0,8)
      (0,1,8)--(1,1,8)--(1,0,8)
      (1,0,0)--(1,1,0)--(0,1,0)
      (1,1,0)--(1,1,8);

  \end{scope}
  \node at (2.1,11.75) {$\mathcal P_5$};
  \node at (0,-.47) {(a) front};
  \node at (4.2,-.47) {(b) rear faces};
\end{tikzpicture}%
}%
\endgroup
  \caption{Two views of $\mathcal{P}_5$.}
  \label{fig:p5}
  \end{minipage}%
  \hfill
  \begin{minipage}[t]{.48\linewidth}
    \centering
  \begingroup
% Input-ready figure for p6; include inside a figure environment.
\definecolor{pigment}{HTML}{53677A}
\colorlet{bodyLeft}{pigment!52}
\colorlet{bodyRight}{pigment!36}
\colorlet{bodyTop}{pigment!44}
\colorlet{bumpA}{pigment!72}
\colorlet{bumpB}{pigment!62}
\colorlet{bumpC}{pigment!45}
\colorlet{bumpD}{pigment!68}
\colorlet{nickA}{pigment!27}
\colorlet{nickB}{pigment!39}
\colorlet{nickC}{pigment!17}
\colorlet{nickD}{pigment!25}
\colorlet{floorFar}{pigment!65}
\colorlet{floorNear}{pigment!28}
\tikzset{edge/.style={draw=black!80,line width=.45pt,line join=round},
  hidden/.style={draw=black!43,line width=.30pt,
    dash pattern=on 1.7pt off 1.5pt,line join=round},
  ghost/.style={draw=black!42,line width=.32pt,
    dash pattern=on .7pt off 1.45pt,line join=round}}
% Square-pyramidal base .30 by .30, apex depth .09.
% Feature on x=0, viewed from negative x and negative y.
\newcommand{\Xfeature}[2]{%
  \pgfmathsetmacro{\lo}{#1+.35}
  \pgfmathsetmacro{\mid}{#1+.5}
  \pgfmathsetmacro{\hi}{#1+.65}
  \pgfmathsetmacro{\dep}{-.09*(#2)}
  \ifnum#2=1
    \path[edge,fill=bumpA] (0,.35,\lo)--(0,.65,\lo)--(\dep,.5,\mid)--cycle;
    \path[edge,fill=bumpB] (0,.65,\lo)--(0,.65,\hi)--(\dep,.5,\mid)--cycle;
    \path[edge,fill=bumpC] (0,.65,\hi)--(0,.35,\hi)--(\dep,.5,\mid)--cycle;
    \path[edge,fill=bumpD] (0,.35,\hi)--(0,.35,\lo)--(\dep,.5,\mid)--cycle;
  \else
    \pgfmathsetmacro{\lo}{#1+.35}
    \pgfmathsetmacro{\hi}{#1+.65}
    \def\dep{.09}
    \begin{scope}
      \clip (0,.35,\lo)--(0,.65,\lo)--(0,.65,\hi)--(0,.35,\hi)--cycle;
      \path[fill=nickA] (0,.35,\lo)--(0,.65,\lo)--(\dep,.5,\mid)--cycle;
      \path[fill=nickB] (0,.65,\lo)--(0,.65,\hi)--(\dep,.5,\mid)--cycle;
      \path[fill=nickC] (0,.65,\hi)--(0,.35,\hi)--(\dep,.5,\mid)--cycle;
      \path[fill=nickD] (0,.35,\hi)--(0,.35,\lo)--(\dep,.5,\mid)--cycle;
      \draw[black!65,line width=.3pt] (\dep,.5,\mid)--(0,.35,\lo)
        (\dep,.5,\mid)--(0,.65,\lo)
        (\dep,.5,\mid)--(0,.65,\hi)
        (\dep,.5,\mid)--(0,.35,\hi);
    \end{scope}
    \draw[edge] (0,.35,\lo)--(0,.65,\lo)--(0,.65,\hi)--(0,.35,\hi)--cycle;
  \fi
}
% Feature on y=0.
\newcommand{\Yfeature}[2]{%
  \pgfmathsetmacro{\lo}{#1+.35}
  \pgfmathsetmacro{\mid}{#1+.5}
  \pgfmathsetmacro{\hi}{#1+.65}
  \pgfmathsetmacro{\dep}{-.09*(#2)}
  \ifnum#2=1
    \path[edge,fill=bumpA] (.35,0,\lo)--(.65,0,\lo)--(.5,\dep,\mid)--cycle;
    \path[edge,fill=bumpB] (.65,0,\lo)--(.65,0,\hi)--(.5,\dep,\mid)--cycle;
    \path[edge,fill=bumpC] (.65,0,\hi)--(.35,0,\hi)--(.5,\dep,\mid)--cycle;
    \path[edge,fill=bumpD] (.35,0,\hi)--(.35,0,\lo)--(.5,\dep,\mid)--cycle;
  \else
    \pgfmathsetmacro{\lo}{#1+.35}
    \pgfmathsetmacro{\hi}{#1+.65}
    \def\dep{.09}
    \begin{scope}
      \clip (.35,0,\lo)--(.65,0,\lo)--(.65,0,\hi)--(.35,0,\hi)--cycle;
      \path[fill=nickA] (.35,0,\lo)--(.65,0,\lo)--(.5,\dep,\mid)--cycle;
      \path[fill=nickB] (.65,0,\lo)--(.65,0,\hi)--(.5,\dep,\mid)--cycle;
      \path[fill=nickC] (.65,0,\hi)--(.35,0,\hi)--(.5,\dep,\mid)--cycle;
      \path[fill=nickD] (.35,0,\hi)--(.35,0,\lo)--(.5,\dep,\mid)--cycle;
      \draw[black!65,line width=.3pt] (.5,\dep,\mid)--(.35,0,\lo)
        (.5,\dep,\mid)--(.65,0,\lo)
        (.5,\dep,\mid)--(.65,0,\hi)
        (.5,\dep,\mid)--(.35,0,\hi);
    \end{scope}
    \draw[edge] (.35,0,\lo)--(.65,0,\lo)--(.65,0,\hi)--(.35,0,\hi)--cycle;
  \fi
}

\newcommand{\prismbody}{%
  \path[edge,fill=bodyLeft] (0,0,0)--(0,1,0)--(0,1,8)--(0,0,8)--cycle;
  \path[edge,fill=bodyRight] (0,0,0)--(1,0,0)--(1,0,8)--(0,0,8)--cycle;
  \path[edge,fill=bodyTop] (0,0,8)--(1,0,8)--(1,1,8)--(0,1,8)--cycle;
  \foreach \level in {1,2,3,4,5,6,7}{%
    \draw[hidden] (0,1,\level)--(0,0,\level)--(1,0,\level);
  }
}

\resizebox{.715\linewidth}{!}{%
\begin{tikzpicture}[font=\small]
  \begin{scope}[x={(1.08253cm,.625cm)},y={(-1.08253cm,.625cm)},z={(0cm,1.25cm)}]
    \prismbody
    \foreach \j/\sgn in {4/-1,5/-1}{\Xfeature{\j}{\sgn}}
    \foreach \j/\sgn in {0/-1,1/-1,6/-1,7/-1}{\Yfeature{\j}{\sgn}}
  \end{scope}
  \begin{scope}[shift={(4.2cm,0)},
      x={(1.08253cm,.625cm)},y={(-1.08253cm,.625cm)},z={(0cm,1.25cm)}]
    % The bottom face is the floor of the open shell.
    \shade[draw=black!70,line width=.45pt,top color=floorFar,bottom color=floorNear] (0,0,0)--(1,0,0)--(1,1,0)--(0,1,0)--cycle;
    % Actual rear faces: x=1 and y=1.
    \path[edge,fill=bodyRight] (1,0,0)--(1,1,0)--(1,1,8)--(1,0,8)--cycle;
    \path[edge,fill=bodyLeft] (0,1,0)--(1,1,0)--(1,1,8)--(0,1,8)--cycle;
    \path[fill=bodyTop,opacity=.3]
      (0,0,8)--(1,0,8)--(1,1,8)--(0,1,8)--cycle;
    \foreach \level in {1,2,3,4,5,6,7}{%
      \draw[hidden] (1,0,\level)--(1,1,\level)--(0,1,\level);
    }
    % Mirrored coordinates draw features on the physical rear faces.
    \begin{scope}[shift={(1,0,0)},x={(-1.08253cm,-.625cm)}]
      \foreach \j/\sgn in {1/1,2/1,7/1}{\Xfeature{\j}{\sgn}}
    \end{scope}
    \begin{scope}[shift={(0,1,0)},y={(1.08253cm,-.625cm)}]
      \foreach \j/\sgn in {3/1,4/1}{\Yfeature{\j}{\sgn}}
    \end{scope}
    % The dotted outline records the removed front walls.
    \draw[ghost]
      (0,0,0)--(0,0,8)--(0,1,8)--(0,1,0)--cycle
      (0,0,0)--(1,0,0)--(1,0,8)--(0,0,8)
      (0,1,8)--(1,1,8)--(1,0,8)
      (1,0,0)--(1,1,0)--(0,1,0)
      (1,1,0)--(1,1,8);

  \end{scope}
  \node at (2.1,11.75) {$\mathcal P_6$};
  \node at (0,-.47) {(a) front};
  \node at (4.2,-.47) {(b) rear faces};
\end{tikzpicture}%
}%
\endgroup
  \caption{Two views of $\mathcal{P}_6$.}
  \label{fig:p6}
  \end{minipage}
\end{figure}

\paragraph*{$2\mathcal{P}_1$: H=2, h=8.}
Tile description:
\[
((-1,-1,1,1)^4, (-1,-1,-1,1)^2, (1,-1,-1,1)^2)
\]

This tile has $d = 4.$

\begin{figure}[!htb]
  \centering
  \begingroup
% Input-ready figure for 2p1; include inside a figure environment.
\definecolor{pigment}{HTML}{53677A}
\colorlet{bodyLeft}{pigment!52}
\colorlet{bodyRight}{pigment!36}
\colorlet{bodyTop}{pigment!44}
\colorlet{bumpA}{pigment!72}
\colorlet{bumpB}{pigment!62}
\colorlet{bumpC}{pigment!45}
\colorlet{bumpD}{pigment!68}
\colorlet{nickA}{pigment!27}
\colorlet{nickB}{pigment!39}
\colorlet{nickC}{pigment!17}
\colorlet{nickD}{pigment!25}
\colorlet{floorFar}{pigment!65}
\colorlet{floorNear}{pigment!28}
\tikzset{edge/.style={draw=black!80,line width=.45pt,line join=round},
  hidden/.style={draw=black!43,line width=.30pt,
    dash pattern=on 1.7pt off 1.5pt,line join=round},
  ghost/.style={draw=black!42,line width=.32pt,
    dash pattern=on .7pt off 1.45pt,line join=round}}
% Square-pyramidal base .30 by .30, apex depth .09.
% Feature on x=0, viewed from negative x and negative y.
\newcommand{\Xfeature}[2]{%
  \pgfmathsetmacro{\lo}{#1+.35}
  \pgfmathsetmacro{\mid}{#1+.5}
  \pgfmathsetmacro{\hi}{#1+.65}
  \pgfmathsetmacro{\dep}{-.09*(#2)}
  \ifnum#2=1
    \path[edge,fill=bumpA] (0,.35,\lo)--(0,.65,\lo)--(\dep,.5,\mid)--cycle;
    \path[edge,fill=bumpB] (0,.65,\lo)--(0,.65,\hi)--(\dep,.5,\mid)--cycle;
    \path[edge,fill=bumpC] (0,.65,\hi)--(0,.35,\hi)--(\dep,.5,\mid)--cycle;
    \path[edge,fill=bumpD] (0,.35,\hi)--(0,.35,\lo)--(\dep,.5,\mid)--cycle;
  \else
    \pgfmathsetmacro{\lo}{#1+.35}
    \pgfmathsetmacro{\hi}{#1+.65}
    \def\dep{.09}
    \begin{scope}
      \clip (0,.35,\lo)--(0,.65,\lo)--(0,.65,\hi)--(0,.35,\hi)--cycle;
      \path[fill=nickA] (0,.35,\lo)--(0,.65,\lo)--(\dep,.5,\mid)--cycle;
      \path[fill=nickB] (0,.65,\lo)--(0,.65,\hi)--(\dep,.5,\mid)--cycle;
      \path[fill=nickC] (0,.65,\hi)--(0,.35,\hi)--(\dep,.5,\mid)--cycle;
      \path[fill=nickD] (0,.35,\hi)--(0,.35,\lo)--(\dep,.5,\mid)--cycle;
      \draw[black!65,line width=.3pt] (\dep,.5,\mid)--(0,.35,\lo)
        (\dep,.5,\mid)--(0,.65,\lo)
        (\dep,.5,\mid)--(0,.65,\hi)
        (\dep,.5,\mid)--(0,.35,\hi);
    \end{scope}
    \draw[edge] (0,.35,\lo)--(0,.65,\lo)--(0,.65,\hi)--(0,.35,\hi)--cycle;
  \fi
}
% Feature on y=0.
\newcommand{\Yfeature}[2]{%
  \pgfmathsetmacro{\lo}{#1+.35}
  \pgfmathsetmacro{\mid}{#1+.5}
  \pgfmathsetmacro{\hi}{#1+.65}
  \pgfmathsetmacro{\dep}{-.09*(#2)}
  \ifnum#2=1
    \path[edge,fill=bumpA] (.35,0,\lo)--(.65,0,\lo)--(.5,\dep,\mid)--cycle;
    \path[edge,fill=bumpB] (.65,0,\lo)--(.65,0,\hi)--(.5,\dep,\mid)--cycle;
    \path[edge,fill=bumpC] (.65,0,\hi)--(.35,0,\hi)--(.5,\dep,\mid)--cycle;
    \path[edge,fill=bumpD] (.35,0,\hi)--(.35,0,\lo)--(.5,\dep,\mid)--cycle;
  \else
    \pgfmathsetmacro{\lo}{#1+.35}
    \pgfmathsetmacro{\hi}{#1+.65}
    \def\dep{.09}
    \begin{scope}
      \clip (.35,0,\lo)--(.65,0,\lo)--(.65,0,\hi)--(.35,0,\hi)--cycle;
      \path[fill=nickA] (.35,0,\lo)--(.65,0,\lo)--(.5,\dep,\mid)--cycle;
      \path[fill=nickB] (.65,0,\lo)--(.65,0,\hi)--(.5,\dep,\mid)--cycle;
      \path[fill=nickC] (.65,0,\hi)--(.35,0,\hi)--(.5,\dep,\mid)--cycle;
      \path[fill=nickD] (.35,0,\hi)--(.35,0,\lo)--(.5,\dep,\mid)--cycle;
      \draw[black!65,line width=.3pt] (.5,\dep,\mid)--(.35,0,\lo)
        (.5,\dep,\mid)--(.65,0,\lo)
        (.5,\dep,\mid)--(.65,0,\hi)
        (.5,\dep,\mid)--(.35,0,\hi);
    \end{scope}
    \draw[edge] (.35,0,\lo)--(.65,0,\lo)--(.65,0,\hi)--(.35,0,\hi)--cycle;
  \fi
}

\newcommand{\prismbody}{%
  \path[edge,fill=bodyLeft] (0,0,0)--(0,1,0)--(0,1,8)--(0,0,8)--cycle;
  \path[edge,fill=bodyRight] (0,0,0)--(1,0,0)--(1,0,8)--(0,0,8)--cycle;
  \path[edge,fill=bodyTop] (0,0,8)--(1,0,8)--(1,1,8)--(0,1,8)--cycle;
  \foreach \level in {1,2,3,4,5,6,7}{%
    \draw[hidden] (0,1,\level)--(0,0,\level)--(1,0,\level);
  }
}

\resizebox{.486\linewidth}{!}{%
\begin{tikzpicture}[font=\small]
  \begin{scope}[x={(1.08253cm,.625cm)},y={(-1.08253cm,.625cm)},z={(0cm,1.25cm)}]
    \prismbody
    \foreach \j/\sgn in {0/1,1/1,2/1,3/1,4/1,5/1,6/1,7/1}{\Xfeature{\j}{\sgn}}
    \foreach \j/\sgn in {0/-1,1/-1,2/-1,3/-1,4/-1,5/-1,6/1,7/1}{\Yfeature{\j}{\sgn}}
  \end{scope}
  \begin{scope}[shift={(4.2cm,0)},
      x={(1.08253cm,.625cm)},y={(-1.08253cm,.625cm)},z={(0cm,1.25cm)}]
    % The bottom face is the floor of the open shell.
    \shade[draw=black!70,line width=.45pt,top color=floorFar,bottom color=floorNear] (0,0,0)--(1,0,0)--(1,1,0)--(0,1,0)--cycle;
    % Actual rear faces: x=1 and y=1.
    \path[edge,fill=bodyRight] (1,0,0)--(1,1,0)--(1,1,8)--(1,0,8)--cycle;
    \path[edge,fill=bodyLeft] (0,1,0)--(1,1,0)--(1,1,8)--(0,1,8)--cycle;
    \path[fill=bodyTop,opacity=.3]
      (0,0,8)--(1,0,8)--(1,1,8)--(0,1,8)--cycle;
    \foreach \level in {1,2,3,4,5,6,7}{%
      \draw[hidden] (1,0,\level)--(1,1,\level)--(0,1,\level);
    }
    % Mirrored coordinates draw features on the physical rear faces.
    \begin{scope}[shift={(1,0,0)},x={(-1.08253cm,-.625cm)}]
      \foreach \j/\sgn in {0/-1,1/-1,2/-1,3/-1,4/-1,5/-1,6/-1,7/-1}{\Xfeature{\j}{\sgn}}
    \end{scope}
    \begin{scope}[shift={(0,1,0)},y={(1.08253cm,-.625cm)}]
      \foreach \j/\sgn in {0/1,1/1,2/1,3/1,4/-1,5/-1,6/-1,7/-1}{\Yfeature{\j}{\sgn}}
    \end{scope}
    % The dotted outline records the removed front walls.
    \draw[ghost]
      (0,0,0)--(0,0,8)--(0,1,8)--(0,1,0)--cycle
      (0,0,0)--(1,0,0)--(1,0,8)--(0,0,8)
      (0,1,8)--(1,1,8)--(1,0,8)
      (1,0,0)--(1,1,0)--(0,1,0)
      (1,1,0)--(1,1,8);

  \end{scope}
  \node at (2.1,11.75) {$2\mathcal P_1$};
  \node at (0,-.47) {(a) front};
  \node at (4.2,-.47) {(b) rear faces};
\end{tikzpicture}%
}%
\endgroup
  \caption{Two views of $2\mathcal{P}_1$.}
  \label{fig:2p1}
\end{figure}
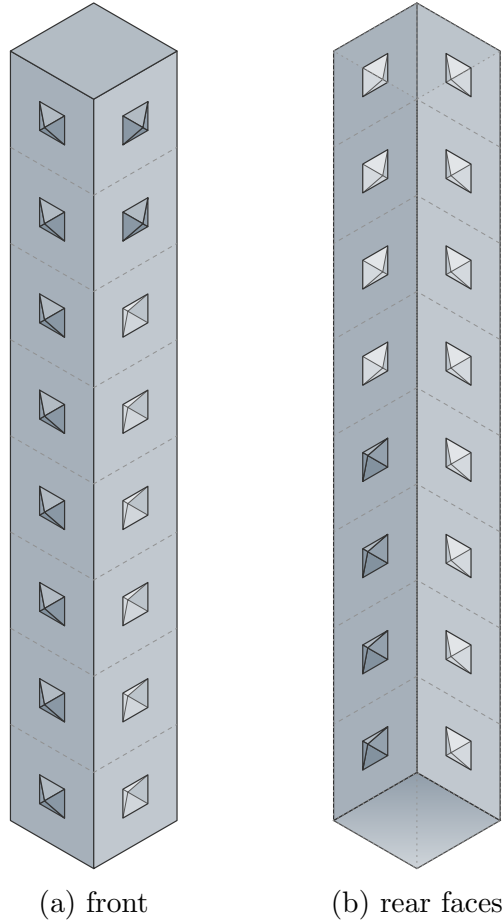

This prism can actually be obtained by "doubling" the first "$\textbf{H=2, h=4}$" example (after a suitable transformation of the prism). Here by doubling we mean that each layer has been doubled, or repeated twice (raised to second power in canonical description). 

Generalizing this, we introduce the following definition:

\begin{defn}
    We say that a prism $p\mathcal{H}$ is obtained by \textit{p-fattening} prism $\mathcal{H}$ if its canonical description is equal to that of $\mathcal{H}$ with every quadruple raised to the $p^{th}$ power.
\end{defn}

\begin{figure}[htbp]
  \centering
  \includegraphics[width=.72\linewidth]{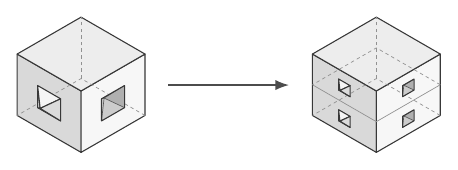}
  \caption{One height-one layer (left) and its $2$-fattening into two layers of height $\frac12$ each (right).}
  \label{fig:layer-doubling}
\end{figure}

This finding gave way to the following, easy, proposition.

\begin{prop}
    If a notched disbalanced prism $\mathcal{H}$ has Heesch number $k,$ then a prism that is obtained by $p$-fattening $\mathcal{H},$ $p\mathcal{H}$, has a finite Heesch number that is at least $k.$
\end{prop}

\begin{proof}
    To see why this is true, notice that $p\mathcal{H}$ is also disbalanced: indeed, if the disbalance (difference really) between the bumps and nicks in $\mathcal{H}$ is $d,$ then the disbalance between bumps and nicks in $p\mathcal{H}$ will be $pd.$ And using an argument analogous to the one first introduced by Amman (see Section \ref{overview}), we have that it has a finite Heesch number (important to note here is that the growth rate of the boundary of the coronas of a square prism is quadratic, while the volume is cubic, which is an easy and tedious task which we left to the reader).
    That the Heesch number of $p\mathcal{H}$ is at least $k$ follows from a simple observation that in some optimal surrounding of $\mathcal{H}$ we can just subdivide each layer into $p$ layers, by imagining bumps and nicks to consist of $p$ parallel smaller copies of themselves. See Figure \ref{fig:layer-doubling} for case $p=2.$ After appropriate transformation this gives  us a surrounding of $p\mathcal{H},$ which finished the proof.
\end{proof}
Observe that the fattened tiles can showcase some new surrounding behaviour: indeed, we do not have to match the blown up faces exactly as they were before -- we could move some of them along the $z$-axis. This gives rise to more possibilities for surrounding such tiles and potentially greater Heesch numbers. And indeed, one can imagine that for a suitable choice of $p$ some new patterns would be possible, but we have no results to support this belief. 
% Note also that the disbalance in these tiles is bigger, which could lead one to believe that they would have even smaller Heesch numbers, but this is not so since the area also grows with the increase of the height of the tiles, and the ratio of excess bumps and the area they can accumulate on remains the same, perhaps suggesting that the Heesch number of each fattened tile should be 2.
As a direct consequence of the previous proposition, we get the following statement:

\begin{stmt}
    Each tile from 6 families of tiles that arise by fattening of the examples shown in this section has a finite Heesch number that is at least $2.$ These families are given by the canonical descriptions of their tiles as follows:
    \[
    \mathcal{F}_1 = \bigl\{((-1,-1,1,1)^p, (-1,-1,-1,1)^p, (1,-1,-1,1)^{2p}) \, : \, p \in \mathbb{N} \bigr\},
    \]
    \[
    \mathcal{F}_2 = \bigl\{((-1,-1,0,1)^p, (0,-1,-1,1)^p, (1,-1,-1,1)^{2p}) \, : \, p \in \mathbb{N} \bigr\},
    \]
    \[
    \mathcal{F}_3 = \bigl\{((-1,-1,1,1)^p, (1,-1,1,1)^{2p}, (1,-1,-1,1)^{5p}) \, : \, p \in \mathbb{N} \bigr\},
    \]
    \[
    \begin{aligned}
    \mathcal{F}_4 = \bigl\{&((-1,-1,-1,-1)^p, (-1,0,-1,-1)^p,
      (-1,1,-1,-1)^p, (0,1,-1,-1)^p,\\
      &\quad (1,1,-1,-1)^p, (1,1,-1,0)^p,
      (1,1,-1,1)^p, (1,1,0,1)^p) \, : \, p \in \mathbb{N} \bigr\},
    \end{aligned}
    \]
    \[
    \begin{aligned}
    \mathcal{F}_5 = \bigl\{&((-1,-1,0,0)^p, (0,0,1,1)^{2p},
      (-1,-1,0,0)^p, (0,-1,0,0)^p,\\
      &\quad (0,0,1,1)^p, (1,0,0,1)^p,
      (0,1,0,0)^p) \, : \, p \in \mathbb{N} \bigr\},
    \end{aligned}
    \]
    \[
    \begin{aligned}
    \mathcal{F}_6 = \bigl\{&((-1,0,0,0)^p, (-1,1,0,0)^p,
      (0,1,0,0)^p, (0,0,1,0)^p,\\
      &\quad (0,0,1,-1)^p, (0,0,0,-1)^p,
      (-1,0,0,0)^p, (-1,1,0,0)^p) \, : \, p \in \mathbb{N} \bigr\}.
    \end{aligned}
    \]
\end{stmt}

\begin{figure}[p]
  \centering
  \includegraphics[height=.92\textheight,keepaspectratio]{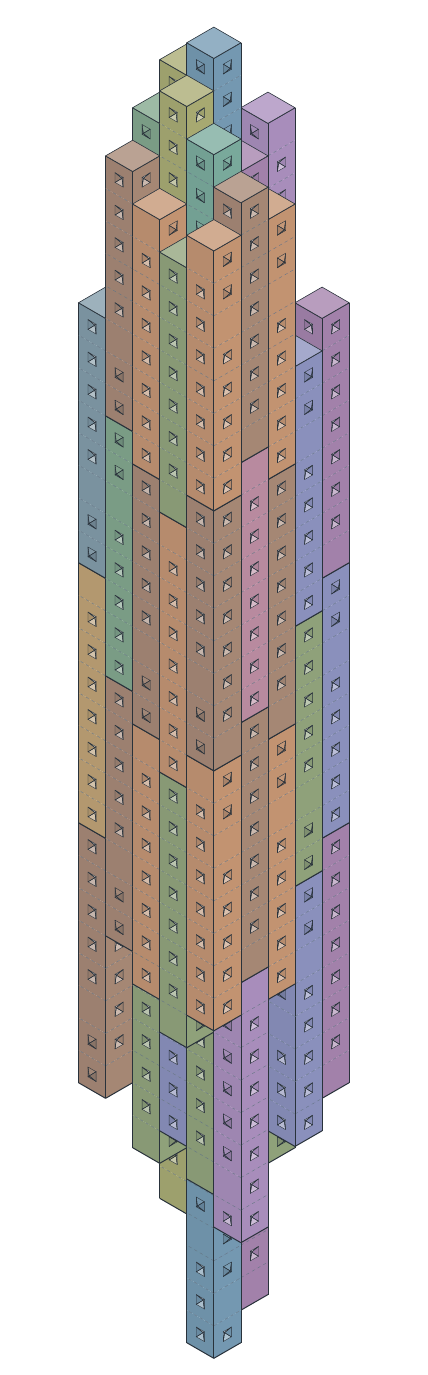}
  \caption{A complete two-corona surrounding of $\mathcal{P}_4$. Colours distinguish the orientations of its prisms, including reflected orientations.}
  \label{fig:p4-two-coronas}
\end{figure}

\begin{figure}[p]
  \centering
  \includegraphics[width=\linewidth,height=.84\textheight,keepaspectratio]{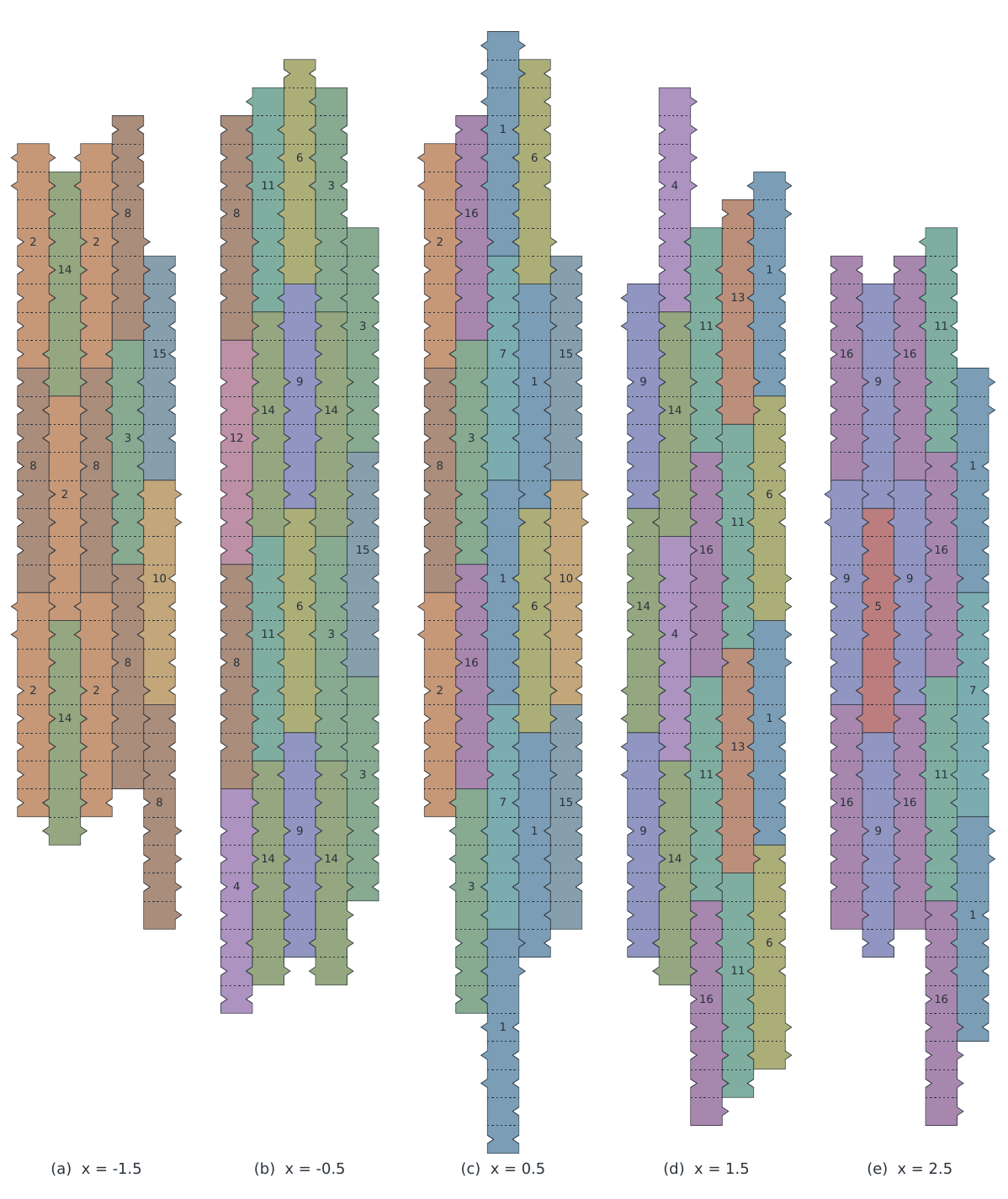}
  \caption{Five vertical $x$--sections of the same surrounding. Tile numbers $1$--$8$ indicate rotations; $9$--$16$ indicate reflected orientations.}
  \label{fig:p4-two-coronas-sections}
\end{figure}

\subsection{Negative results}
We also tried searching in a lot of specific families, which resulted in no findings, thus proving that no notched prism of a certain type exists. Most of these were motivated by the examples shown above, in hopes of generalizing them, or finding some other tiles with those specific restrictions that would have high Heesch numbers. In a few cases it will seem like the restrictions are neverending, but this was done since, as mentioned, sometimes the space of all the posibilities was too large. We list these results below:

\begin{itemize}
    \item The two "$\mathbf{H = 2, h=4}$" examples are the only such;
    \item There are no "$\mathbf{H=3, h=8}$" prisms that are disbalanced;
    \item There are no "$\mathbf{H=2, h=8}$" prisms whose each layer is non--negative (and at least one positive) -- meaning that each entry in that layer is either $0$ or $1$, and such that each layer has at least two flat sides;
    \item There are no "$\mathbf{H=3, h=16}$" prisms such that: all the even layers are mutually congruent, as well as all the odd ones, and each two consecutive layers (really, just the two height one prisms we choose for even and odd layers) can differ only on one side;
    \item There is no "$\mathbf{H=2, h=8}$" prism such that: each layer has at least two flat faces, and one of the layers can have sum equal to $-1$, while the rest are non--negative, and there are two positive layers to ensure the tile is disbalanced (one with sum $+2$ would be too hard for the program);
    \item There is no "$\mathbf{H=3, h=16}$" prism such that: no side of any layer is flat, and the whole lower half is $(1,1,-1,-1);$
    \item There is no "$\mathbf{H=3, h=16}$" prism such that one of four sides is only bumps and the opposite side is all nicks. Additionally, there are no flat sides on the other two sides;
    \item There is no "$\mathbf{H=3, h=16}$" prism such that every second layer is $(-1,-1,-1,-1);$
    \item There is no "$\mathbf{H=3, h=16}$" prism such that every second layer is $(1,1,-1,-1)$, every other layer is non--negative, and there is at least one positive layer;
    \item There are no other "stairy" "$\mathbf{H=2, h=8}$" prisms except for $\mathcal{P}_4$, where by stairy we mean that each side is increasing.
\end{itemize}

\subsection{AI disclosure}

During the preparation of this article a full solution to Heesch's problem in $\mathbb{E}^3$ was announced \cite{tilinglist}. Nudged by the necessity stemming from this sudden development, I resorted to asking ChatGPT 6 to proofread the paper, and it found a mistake in my proof of Lemma \ref{lem}. Being unable to fix it on such short notice (indeed, the majority of the paper was written in a 6--day haze), I opted to write down its proof of the same Lemma, albeit in a more presentable and simplified manner. It is thus that complete proofs (although still in some part based on the original approach by the author) of Lemmas \ref{lem} \ref{band} are completely due to AI. The proofs are considerably hard and make an important part of the paper. I would add that I was surprised by the elegance of the proofs and that I found them quite beautiful, and a valuable addition to the structural study of ratio $q$ tilings. I vouch for having checked them and claim that they are as true as proofs get. I am hopeful that in some future version of the paper I may come up with my own proofs of the said Lemmas, although the ones presented here seem to be the most elegant and simple. That said, the main idea of the paper, the new approach, and all the parts of the proofs (omitting the two mentioned lemmas), as well as the writing of the paper, are original. I acknowledge the use of ChatGPT--6 for proof checking and for grammar checks. I also acknowledge that the figures were made with the help of ChatGPT--6.

\section{Acknowledgements}

I am thankful to Bojan Ba\v si\' c for running and writing all the computational tests mentioned in the paper, as well as for some fruitful discussions during the starting stages of this paper.
My gratitude also goes to Anton Bakker for allowing me to showcase the tile shown in Figure \ref{hes6}.

This work was supported by grant GA\v{C}R 23-04949X, GAUK grant VV--2025--2608/22 and by the
Visegrad Fund.

\bibliographystyle{abbrv}
\bibliography{heesch_references}
\end{document}